\documentclass{article}
\usepackage{maa-monthly}

\usepackage{amsfonts,amssymb,latexsym,amsmath}

\final

\usepackage{graphicx}
\graphicspath{{Graphics/}}

\usepackage{float}

\newtheorem{thm}{Theorem}
\newtheorem{prop}[thm]{Proposition}

\newtheorem{lem}[thm]{Lemma}
\newtheorem{conj}[thm]{Conjecture}

\newtheorem{question}{Question}

\theoremstyle{definition}
\newtheorem{defn}[thm]{Definition}
\newtheorem*{rem}{Remark}

\newcommand{\exact}[1]{\textbf{\textcolor{red}{#1}}}
\newcommand{\mathexact}[1]{\textcolor{red}{#1}}

\newcommand{\fl}[1]{\lfloor#1\rfloor}
\newcommand{\cl}[1]{\lceil #1\rceil}

\newcommand{\Fl}[1]{\left\lfloor#1\right\rfloor}
\newcommand{\fr}[1]{\{#1\}}
\newcommand{\Fr}[1]{\left\{#1\right\}}

\newcommand{\NN}{\mathbb{N}}
\newcommand{\ZZ}{\mathbb{Z}}

\newcommand{\QQ}{\mathbb{Q}}

\newcommand{\lgt}{\log_{10}}

\begin{document}

\title{The Surprising Accuracy of Benford's Law in Mathematics}

\markright{The Surprising Accuracy of Benford's Law}

\author{Zhaodong Cai, Matthew Faust, A.~J. Hildebrand, Junxian Li,\\
and Yuan Zhang}

\maketitle

\begin{abstract}
Benford's law is an empirical ``law'' governing the frequency of leading
digits in numerical data sets.  Surprisingly, for mathematical sequences the
predictions derived from it can be uncannily accurate.  For example, among
the first billion powers of $2$, exactly $301029995$ begin with digit 1,
while the Benford prediction for this count is
$10^9\log_{10}2=301029995.66\dots$.  Similar ``perfect hits'' can be
observed in other instances, such as the digit $1$ and $2$ counts for the first
billion powers of $3$.  We prove results that explain many, but not all, of 
these surprising accuracies, and we relate the observed behavior to classical
results in Diophantine approximation as well as recent deep conjectures in
this area. 
\end{abstract}

\section{Introduction.}
\label{sec:intro}

\emph{Benford's law} is the empirical observation that leading
digits in many real-world data sets tend to follow the \emph{Benford
distribution}, depicted in Figure \ref{fig:benford-distribution} and
given by 
\begin{equation}
\label{eq:benford} 
P(\text{first digit is $d$}) 
=P(d)=
\log_{10}\left(1+\frac1d\right),\quad
d=1,2,\dots,9.  
\end{equation}

\begin{figure}[H]
\begin{center}
\includegraphics[width=.6\textwidth]{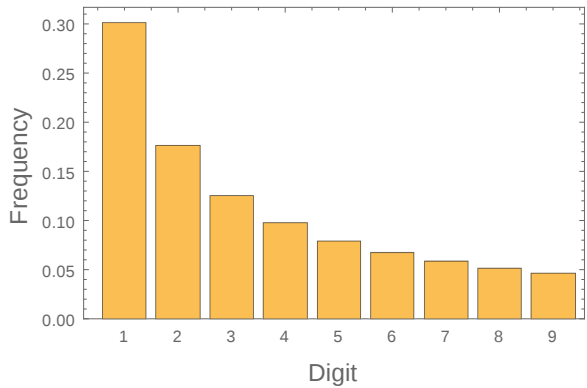}
\end{center}
\caption{The Benford distribution, $P(d)=\log_{10}(1+1/d)$.}
\label{fig:benford-distribution}
\end{figure}
Thus, in a data set following Benford's law, approximately 
$\log_{10}2\approx 30.1\%$ of the numbers begin with digit $1$, 
approximately $\log_{10}(3/2)\approx 17.6\%$ begin with digit $2$, 
while only around $\log_{10}(10/9)\approx 4.6\%$ begin with digit $9$.

Benford's law has been found to be a good match for a wide range of real
world data, from populations of cities to accounting data,
and it has been the subject of nearly one thousand articles
(see the online bibliography \cite{benfordonline}),  including several 
\textsc{Monthly} articles (e.g., \cite{hill1995,raimi1976,
ross2011}).  It has also long been known (see, e.g., 
\cite{diaconis1977}) that Benford's law holds for many ``natural''
mathematical sequences with sufficiently fast rate of growth, such as the
Fibonacci numbers, the powers of $2$, and the sequence of factorials.  In
this context, saying that Benford's law holds is usually
understood to mean that, for each digit $d\in\{1,2,\dots,9\}$, the
proportion of terms beginning with digit $d$ among the first $N$ terms of
the sequence converges to the Benford frequency $P(d)$ given by
\eqref{eq:benford}, as $N\to\infty$.

\subsection{How accurate is Benford's law?} Given a sequence such as the
powers of $2$, Benford's law predicts that, among
the first $N$ terms of the sequence, approximately $N\log_{10}(1+1/d)$
begin  with digit $d$, for each $d\in\{1,2,\dots,9\}$.  How good are these
approximations?  A  natural benchmark is a random model: Imagine the
sequence of leading digits were generated randomly by repeated throws of
a 9-sided die with faces marked $1,2,\dots,9$, weighted such that face $d$
comes up with the Benford probability $P(d)=\log_{10}(1+1/d)$.  Under these
assumptions, by the central limit theorem the difference between the actual
and predicted digit counts among the first $N$ terms will be roughly of
order $\sqrt{N}$. Thus, in a data set consisting of a billion terms 
(i.e., with $N=10^9$) it would be reasonable to expect errors on the order
of $10,000$. 

Random models of the above type form the basis of numerous conjectures in number
theory, most notably the Riemann hypothesis.  However, there also exist
problems in which, due to additional structure inherent in the problem,  it
is reasonable to expect smaller errors than the squareroot type errors
that are typical for random situations.  Two classic examples of this type are the
circle problem of Gauss and the divisor problem of Dirichlet, which have
been the subject of a recent \textsc{Monthly} article 
\cite{berndt2018}.  In both of these problems the ``correct''
order of the error terms is believed to be $N^{1/4}$. For $N=10^9$, this
would suggest errors on the order of $100$. 

Finally, there are examples in number theory in which the approximation
error, while still exhibiting ``random'' behavior, grows at a logarithmic
rate. One such case is a problem investigated by Hardy and Littlewood 
\cite{hardy-littlewood1921} concerning the number of lattice points in a 
right triangle.

How good are the predictions provided by Benford's law when compared to
such benchmarks?  The surprising answer is that, in many cases, these
predictions  appear to be uncannily accurate---more accurate than 
any of the above benchmarks, and more accurate than even the most
optimistic conjectures would lead one to expect. 
In fact, when we first observed some remarkable coincidences in data we had
compiled for a different project
\cite{mersenne-benford}, 
we thought of 
them as mere flukes.  Later we revisited the problem, approaching it in a
systematic manner, expecting to either confirm the ``fluke'' nature of
these coincidences, or to come up with a simple explanation for them. 

What we found instead was something far more complex, and more interesting,
than any of us had anticipated. Our attempt at getting to the bottom of
some seemingly insignificant numerical coincidences turned into a research
adventure full of surprises and unexpected twists that required unearthing
little-known classical results in Diophantine approximation as well as
drawing on some of the deepest recent work in the area.  In this article we
take the reader along for the ride in this  adventure in mathematical research
and discovery, and we describe the results that came out of this work.

\subsection{Outline of the article.}
The rest of this article is organized as follows.  In Sections
\ref{sec:data1}--\ref{sec:data2} we present the numerical data
alluded to above, we formalize several notions of ``surprising''
accuracy, and we pose three questions suggested by
the numerical observations that will serve as guideposts for our
investigations.  The remainder of the article is devoted to unraveling the
mysteries behind the numerical observations and uncovering,
to the extent possible, the underlying general phenomena. 
We proceed in  three stages, corresponding to three different levels of
sophistication in terms of the mathematical tools used.  The three stages
are largely independent of each other, and they can be read independently.

In the first stage, consisting of Sections \ref{sec:mystery1} and
\ref{sec:mystery2}, we use an entirely elementary approach to settle the
mystery in a particularly interesting special case.  In the second stage,
presented in 
Sections \ref{sec:benford-interval-discrepancy}--\ref{sec:bounded-error},
we draw on results by Ostrowski and Kesten from the mid 20th
century to obtain a general solution to the mystery in the ``bounded
Benford error'' case.  In the third stage, contained in Section
\ref{sec:unbounded-error}, we bring recent groundbreaking and deep
work of J\'{o}zsef Beck to bear on  the remaining---and most
difficult---case, that of an ``unbounded Benford error,'' and we present 
the surprising denouement of the mystery in this case. 

The final section, Section \ref{sec:concluding-remarks}, contains some
concluding remarks on extensions and generalizations of these results
and related
results.

\section{Numerical Evidence: Exhibit A.}
\label{sec:data1}

We begin by presenting some of the numerical data that had spurred our
initial investigations.  Our data consisted of leading digit counts for
the first billion terms of a variety of ``natural'' mathematical
sequences.  Carrying out such large scale computations is a highly
nontrivial task that, among other things, required the use of specialized
C++ libraries for arbitrary precision real number arithmetic. 
The technical details are described in \cite{mersenne-benford}.

Table \ref{table:benford-1bil} shows the actual leading digit counts for
the sequences $\{2^n\}$, $\{3^n\}$, and $\{5^n\}$,
along with the predictions provided by
Benford's law, i.e., $N\log_{10}(1+1/d)$, where $N=10^9$. 


\begin{table}[H]
\caption{%
Predicted versus actual counts of leading digits among the first billion
terms of the sequences $\{2^n\}$, $\{3^n\}$, $\{5^n\}$. 
Entries in \exact{boldface} fall within $\pm1$ of the predicted counts.
}
\label{table:benford-1bil}
\begin{center}
\begin{tabular}{|c||c|c|c|c|}
\hline
Digit & Benford Prediction & $\{2^n\}$ & $\{3^n\}$ &  $\{5^n\}$ 
\\
\hline
\hline
1 & 301029995.66 & \exact{301029995} & \exact{301029995} &
  \exact{301029995}
\\
\hline
2 & 176091259.06 & 176091267 & \exact{176091259} & 176091252
\\
\hline
3 & 124938736.61 & 124938729 & \exact{124938737} & 124938744
\\
\hline
4 & 96910013.01 & \exact{96910014} & 96910012 & \exact{96910013}
\\
\hline
5 & 79181246.05 & 79181253 & \exact{79181247} & 79181239
\\
\hline
6 & 66946789.63 & 66946788 & 66946787 & 66946793
\\
\hline
7 & 57991946.98 & 57991941 & 57991952 & 57991951
\\
\hline
8 & 51152522.45 & 51152528 & 51152520 & 51152519
\\
\hline
9 & 45757490.56 & 45757485 & \exact{45757491} & 45757494
\\
\hline
\end{tabular}
\end{center}
\end{table}
Remarkably, nine out of the $27$ entries in this table 
fall within $\pm1$ of the Benford predictions and are equal to  
the floor or the ceiling of the predicted values.
This is an amazingly good
``hit rate'' for numbers that are on the order of $10^8$.  Of the
remaining $18$ entries, all are within a single digit error of the
predicted value.

As remarkable as these observed coincidences seem to be, one has be careful
before jumping to conclusions.  For example, a ``perfect hit'' observed at
$N=10^9$  might just be a coincidence that does not persist at other values
of $N$.  Such coincidences would not be particularly unusual in case the
errors in the Benford approximations have a slow (e.g., logarithmic) rate
of growth.  

One must also keep in mind Guy's ``strong law of small numbers''
\cite{guy}, which refers to situations in which the ``true'' behavior is
very different from the behavior that can be observed within the computable
range. Such situations are not uncommon in number theory; Guy's paper includes 
several examples.  Could it be that the uncanny accuracy of Benford's
law observed in Table \ref{table:benford-1bil} is just a manifestation of
Guy's ``strong law of small numbers,'' and thus a complete mirage?

\section{Perfect Hits and Bounded Errors.}
\label{sec:perfect-hits}

Motivated by the observations in Table \ref{table:benford-1bil}, we now
formalize several notions of ``surprising'' accuracy of Benford's law.

We begin by introducing some basic notations. We denote by $D(x)$ the
\emph{leading} (i.e., \emph{most significant}) digit of a positive number
$x$, expressed in its standard decimal expansion and ignoring leading
$0$'s; for example, $D(\pi)=D(3.141\dots)=3$ and $D(1/6)=D(0.166\dots)=1$.

We write $\fl{x}$ (respectively, $\cl{x}$) for the
\emph{floor} (respectively, \emph{ceiling})
of a real number $x$, and $\fr{x}=x-\fl{x}$  for its 
\emph{fractional part}.

Given a sequence $\{a_n\}$ of positive real numbers and
a digit $d\in\{1,2,\dots,9\}$, we define 
the associated \emph{leading digit counting function} as 
\begin{align*}
S_d(N,\{a_n\}) &=\#\{n\le N: D(a_n)=d\},
\end{align*}
where, here and in the sequel,
$N$ denotes a positive integer and the notation ``$n\le N$''
means that $n$ runs over the integers $n=1,2,\dots,N$.
We denote the \emph{Benford approximation}, or \emph{Benford prediction},
for the counting function $S_d(N,\{a_n\})$ by 
\begin{align*}
B_d(N) &=NP(d)=N\lgt\left(1+\frac1d\right).
\end{align*}

In terms of these notations, the entries in the second 
column of Table \ref{table:benford-1bil}
are $B_d(10^9)$, $d=1,2,\dots,9$, while those in the three right-most 
columns are $S_d(10^9,\{a^n\})$, $d=1,2,\dots,9$, for $a=2$, $a=3$, and
$a=5$.


\begin{defn}[Perfect Hits and Bounded Errors]
\label{def:perfect-hits}
Let $\{a_n\}$ be a sequence of positive real numbers and let
$d\in\{1,2,\dots,9\}$.  
\begin{itemize}
\item[(i)]
We call the Benford prediction for the leading digit $d$ in the sequence
$\{a_n\}$ a \textbf{perfect hit} if it satisfies either
\begin{equation}
\label{eq:perfect-hit1}
S_d(N,\{a_n\})=\fl{B_d(N)}\quad\text{for all $N\in\NN$},
\end{equation}
or 
\begin{equation}
\label{eq:perfect-hit2}
S_d(N,\{a_n\})=\cl{B_d(N)}\quad\text{for all $N\in\NN$},
\end{equation}
i.e., if the actual leading digit count is \emph{always} equal to 
the predicted count rounded \emph{down} (respectively, \emph{up})
to the nearest integer. 
In the first case we call the Benford prediction 
a \textbf{lower perfect hit}, while in the second case 
we call it an \textbf{upper perfect hit}.

\item[(ii)]
We say that the Benford prediction
for the leading digit $d$ in the sequence $\{a_n\}$ 
has \textbf{bounded error} if there exists a constant $C$ such that
\begin{equation}
\label{eq:bounded-error}
|S_d(N,\{a_n\})-B_d(N)|\le C 
\quad\text{for all $N\in\NN$}.
\end{equation}

\end{itemize}
\end{defn}


\begin{rem}
Define the \emph{Benford error} 
as the difference between the actual and predicted leading digit counts: 
\begin{align}
E_d(N,\{a_n\}) &=S_d(N,\{a_n\})-B_d(N).
\label{eq:benford-error}
\end{align}
Then  the above definitions can  
be restated in terms of the Benford error as follows.
\begin{align}
\label{eq:perfect-hit1-criterion}
\text{lower perfect hit\ }
&\Longleftrightarrow 
-1< E_d(N,\{a_n\})\le 0\quad \text{for all $N\in\NN$,}
\\
\label{eq:perfect-hit2-criterion}
\text{upper perfect hit\ }
&\Longleftrightarrow
\ 0\le E_d(N,\{a_n\})< 1\quad \text{for all $N\in\NN$,}
\\
\text{bounded error\ }
&\Longleftrightarrow 
|E_d(N,\{a_n\})|\le C\quad \text{for some $C$ and all $N\in\NN$.}
\label{eq:bounded-error-criterion}
\end{align}
\end{rem}


As observed above, of the $27$ entries in  Table \ref{table:benford-1bil} 
nine are equal to the Benford prediction rounded up or down to an integer.
Hence, each of these cases represents a \emph{potential} 
perfect hit 
in the sense of Definition \ref{def:perfect-hits}.
This suggests the following questions:

\begin{question}[Perfect Hits]
\label{questionPH}
Which, if any, of the nine \emph{observed}  ``perfect hits'' 
 in Table \ref{table:benford-1bil}
 are ``for real,''  i.e., are instances of a true 
 perfect hit in the sense of Definition \ref{def:perfect-hits}? 
 \end{question}

\begin{question}[Bounded Errors]
\label{questionBE}
Which, if any, of the $27$ entries in Table \ref{table:benford-1bil}
represent cases in which the Benford prediction 
has bounded error?  
\end{question}

In this article we will provide a complete answer to these questions,
not only for the sequences shown in Table \ref{table:benford-1bil}, but for
arbitrary sequences of the form $\{a^n\}$.  We encourage the reader to
guess the answers to these questions before reading on. Suffice it to say
that our own initial guesses turned out to be way off!


\section{Numerical Evidence: Exhibit B.}
\label{sec:data2}

For further insight into the behavior of the Benford approximations,  
it is natural to consider the \emph{distribution} of the 
Benford errors  defined in \eqref{eq:benford-error} as $N$ varies.
Focusing on the sequence $\{2^n\}$, we have computed, for each digit
$d\in\{1,2,\dots,9\}$, the quantities $E_d(N; \{2^n\})$, $N=1,2,\dots,
10^9$,
and plotted a histogram of the distribution of these $10^9$ terms. 
The results, shown in Figure \ref{fig:histograms}, turned out to be quite
unexpected.

\begin{figure}[H]
\begin{center}
\includegraphics[width=0.31\textwidth]{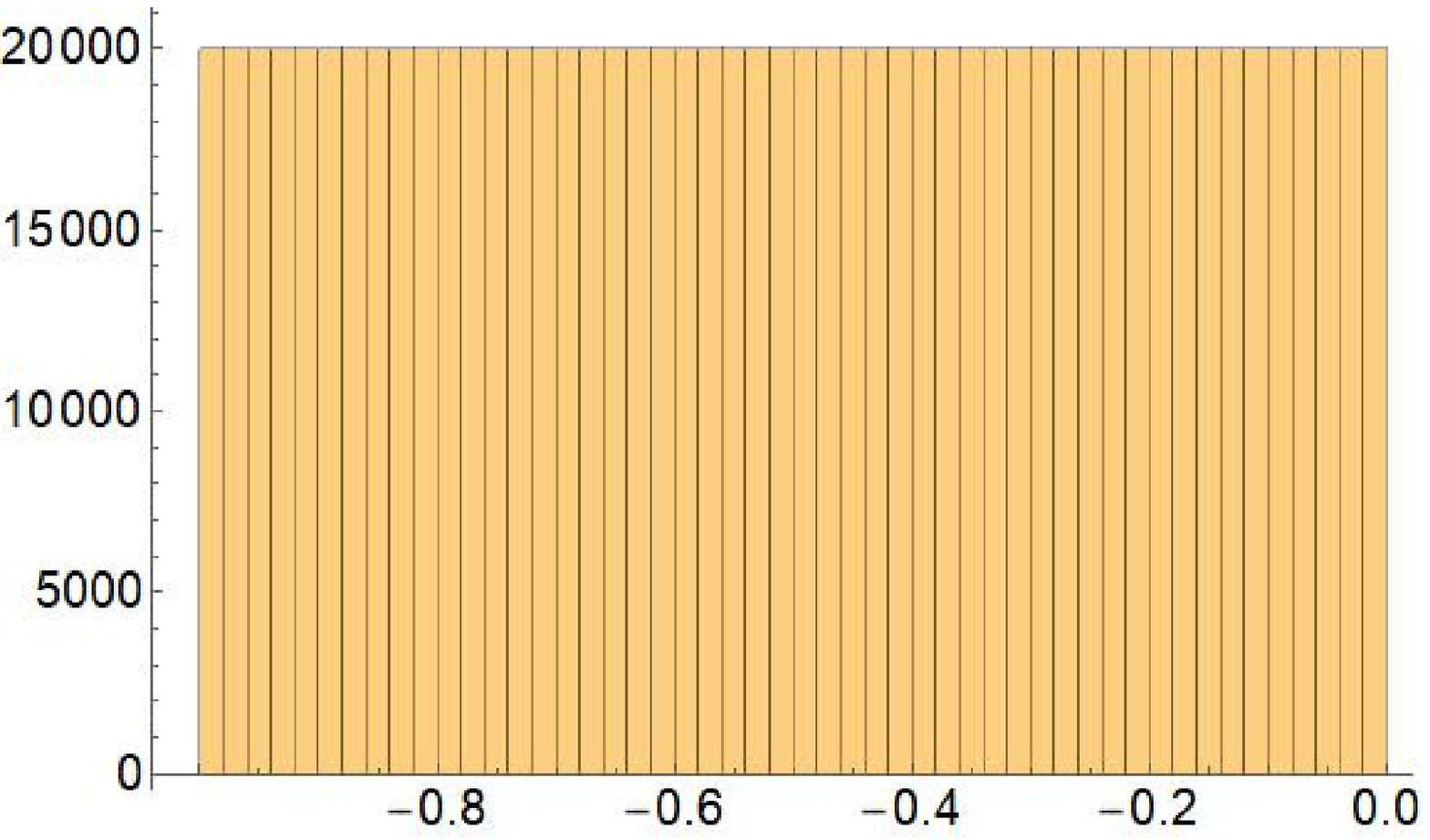}
\hspace{3pt}
\includegraphics[width=0.31\textwidth]{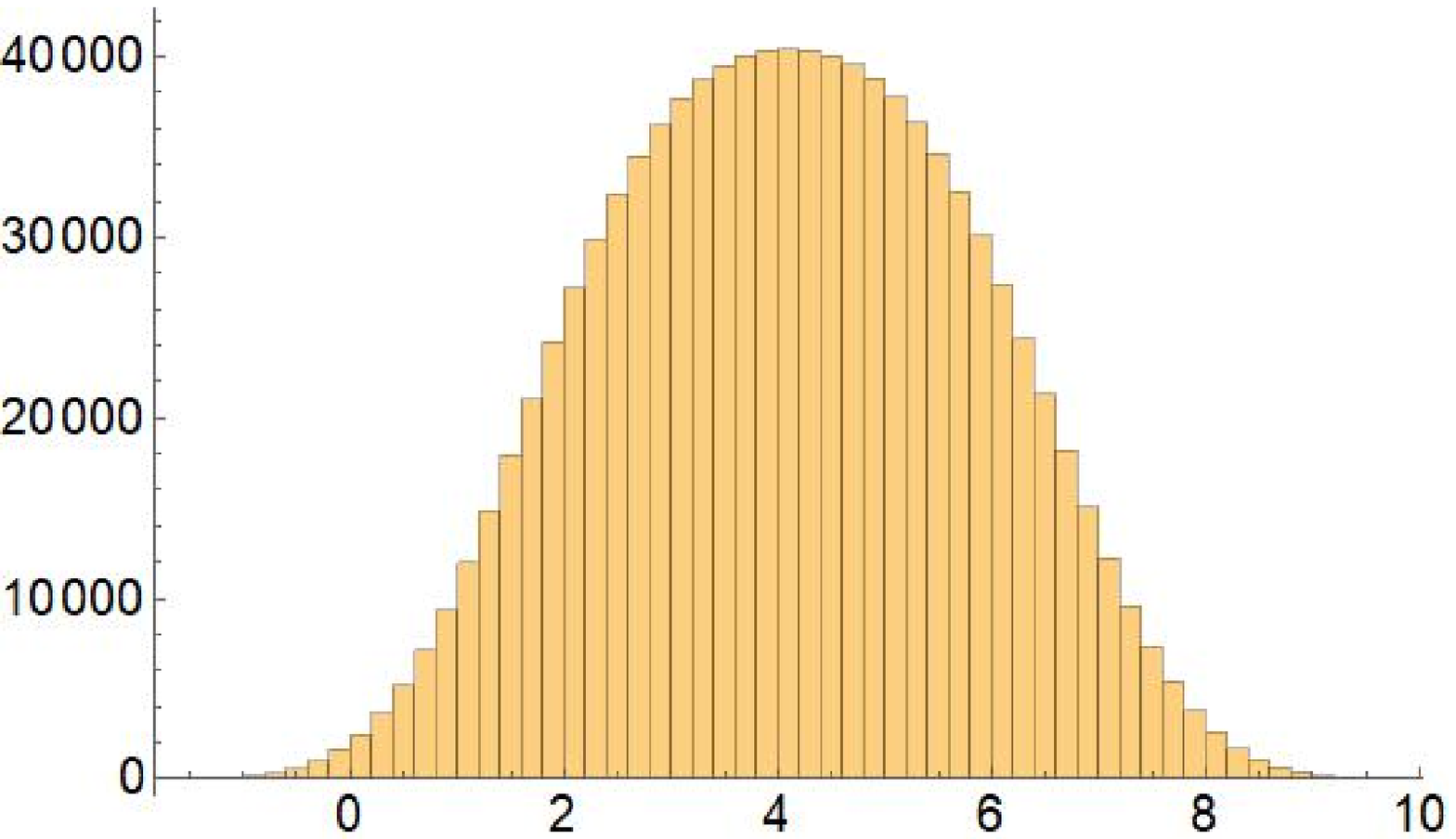}
\hspace{3pt}
\includegraphics[width=0.31\textwidth]{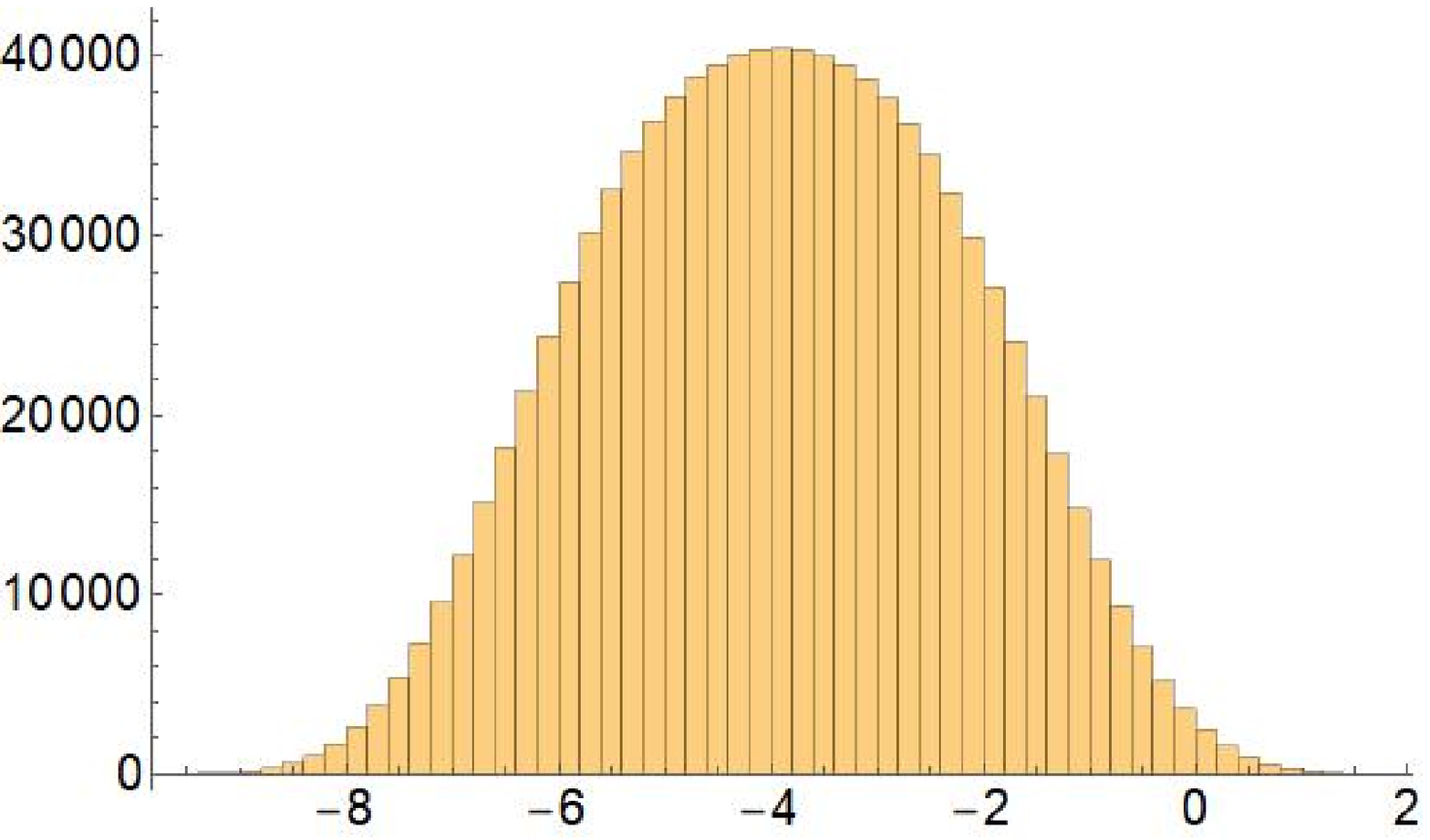}
\\[.5ex]
\includegraphics[width=0.31\textwidth]{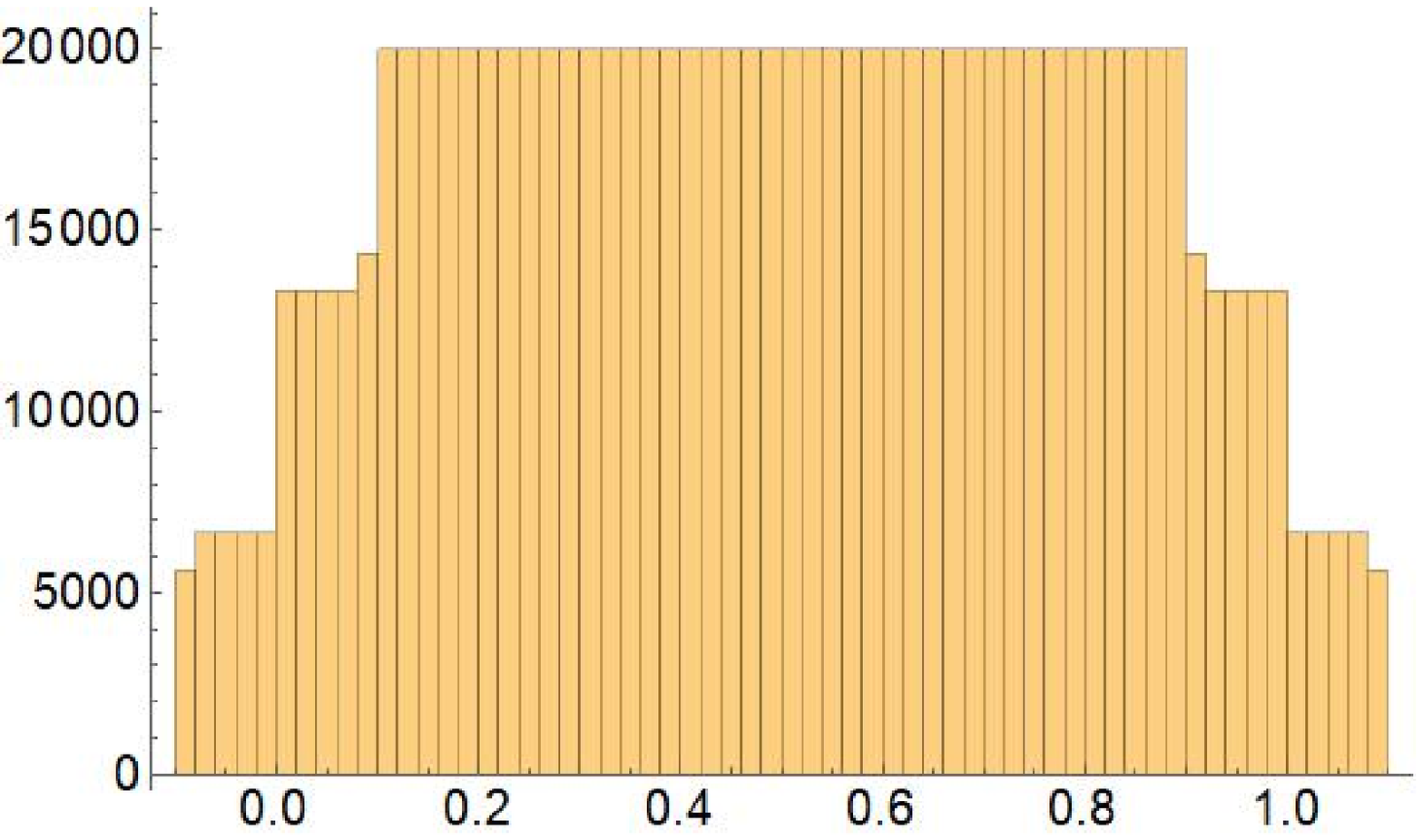}
\hspace{3pt}
\includegraphics[width=0.31\textwidth]{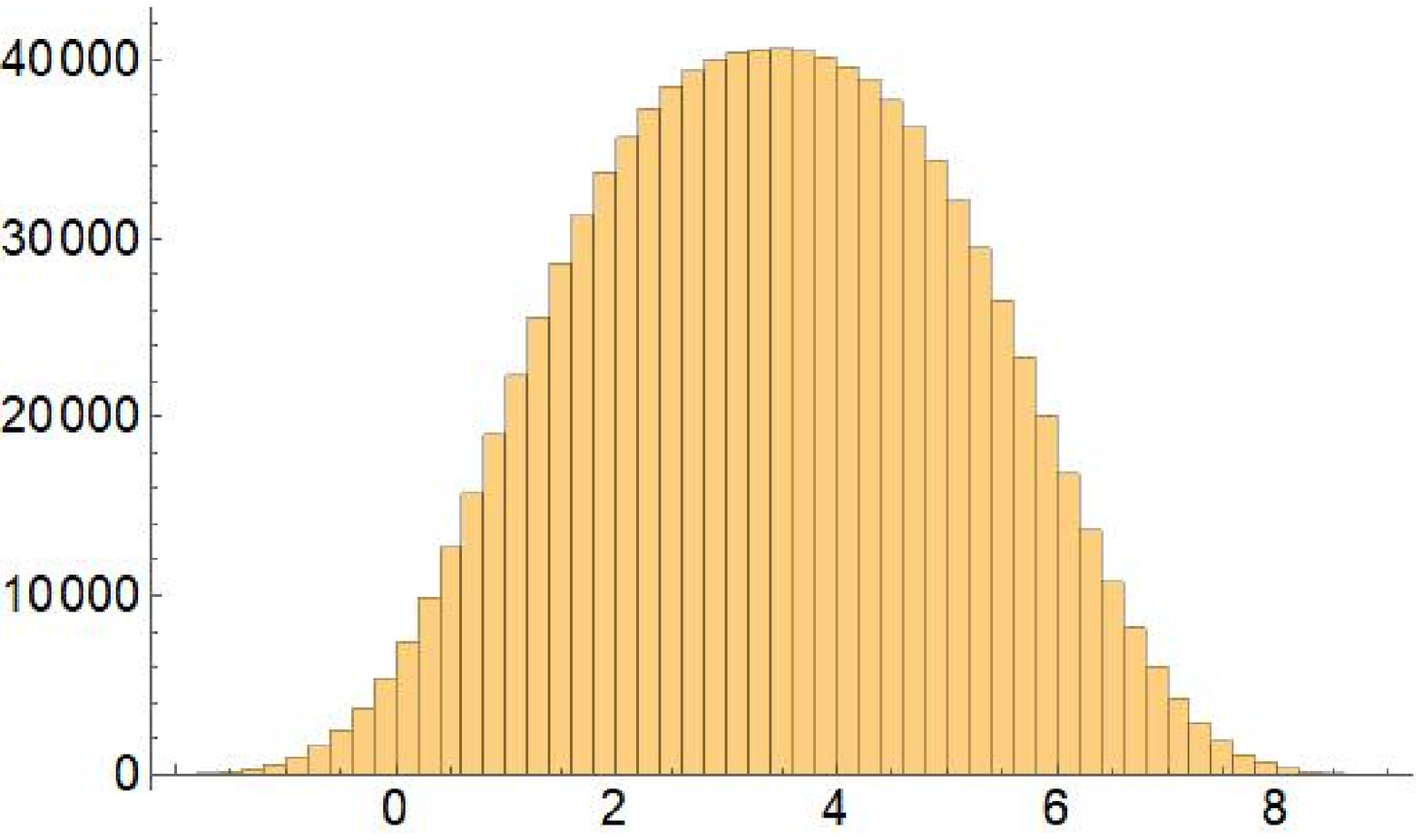}
\hspace{3pt}
\includegraphics[width=0.31\textwidth]{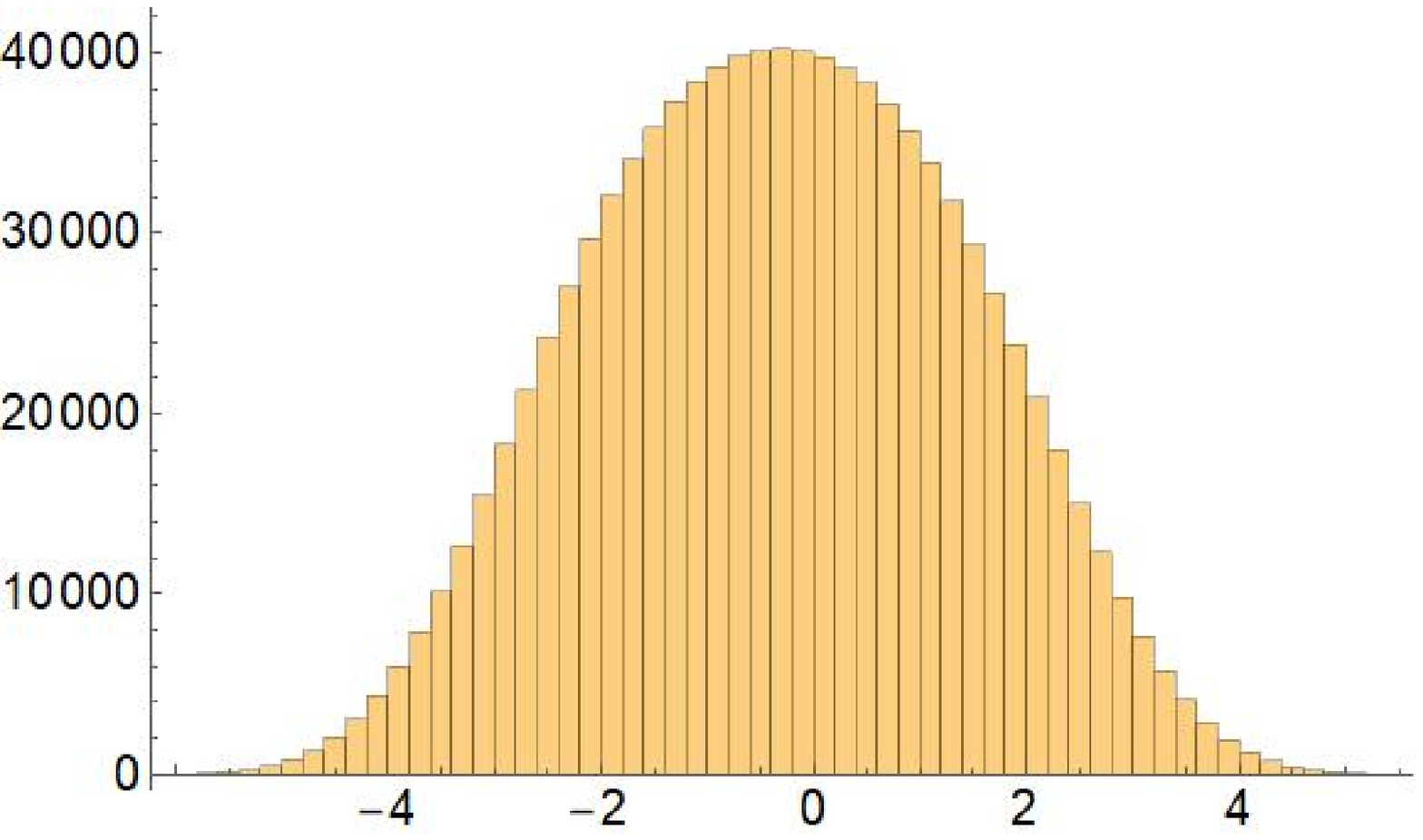}
\\[.5ex]
\includegraphics[width=0.31\textwidth]{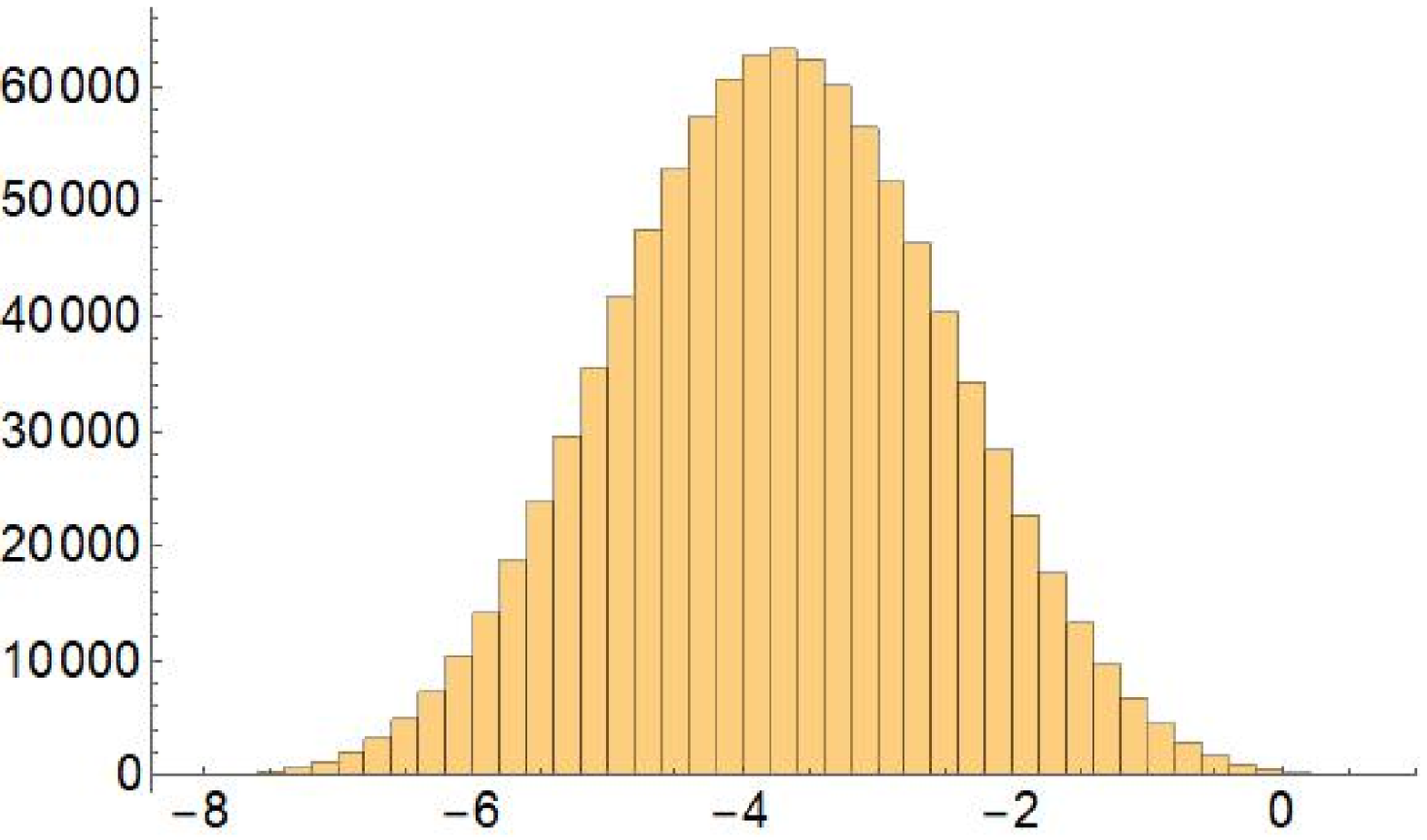}
\hspace{3pt}
\includegraphics[width=0.31\textwidth]{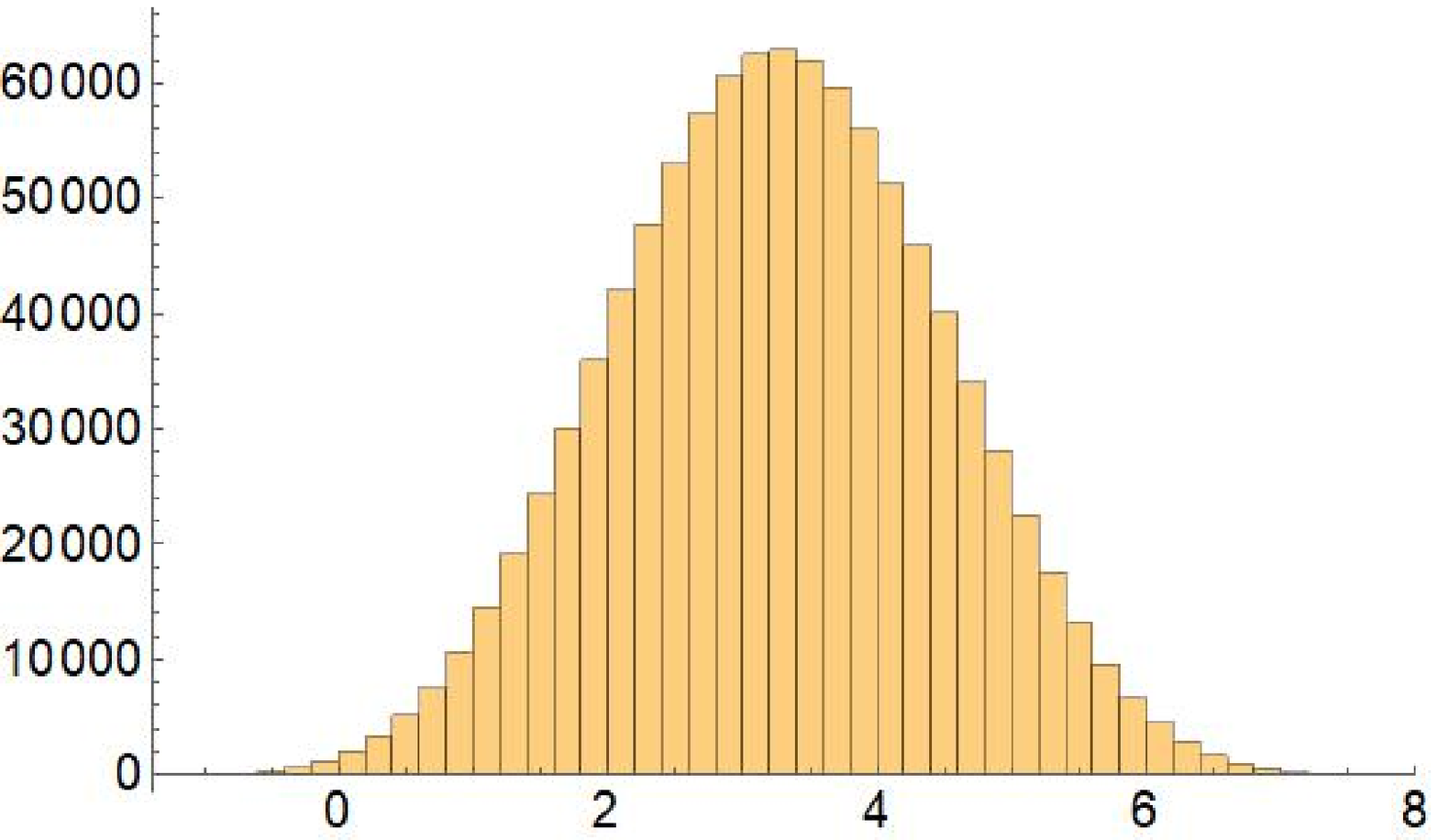}
\hspace{3pt}
\includegraphics[width=0.31\textwidth]{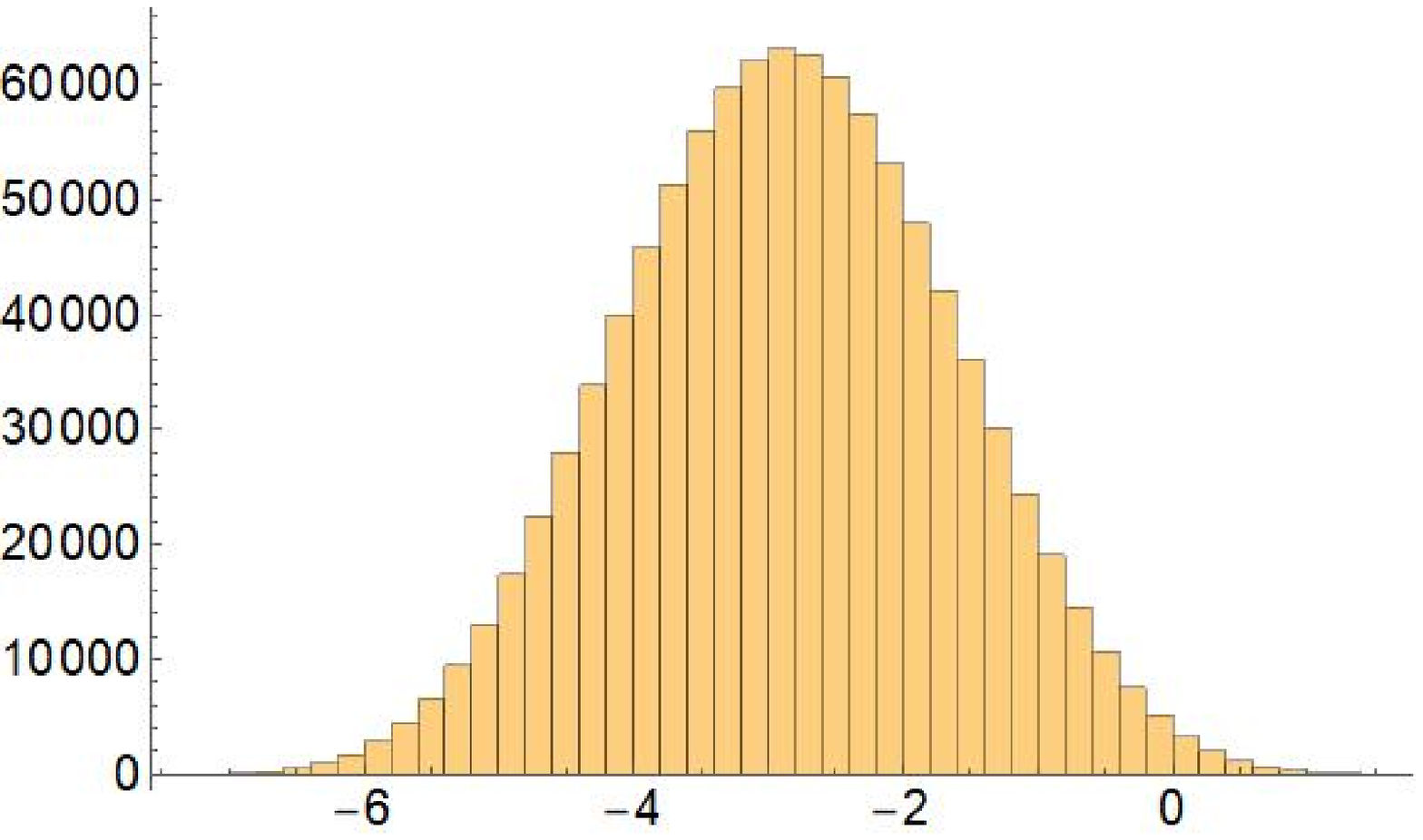}
\caption{Distribution of the Benford errors 
for the sequence $\{2^n\}$, based on the first billion terms of this sequence.
The three rows of histograms show the distributions 
of Benford errors for digits 1--3, 4--6, and 7--9, respectively.
}
\label{fig:histograms}
\end{center}
\end{figure}

The most noticeable, and least surprising, feature in Figure
\ref{fig:histograms}
is the approximately normal shape of seven of the nine distributions shown.
This suggests that the corresponding Benford errors are asymptotically normally
distributed.  The means and standard deviations of these distributions are
in the order of single digits, indicating a logarithmic, or even
sublogarithmic, growth rate. 

The error distribution for digit $1$ (shown in the top left histogram) also has
an easily recognizable shape: It appears to be a uniform distribution
supported on the interval $[-1,0]$. 

By contrast, the error distribution for digit $4$ (shown in the middle left
histogram) does not resemble any familiar distribution and seems to be a
complete mystery.  Unraveling this mystery, and discovering the underlying
general mechanism, has been a key motivation and driving force in our
research; we will describe the results later in this article.  In the
meantime, the reader may ponder the following question, keeping in mind
the possibility of Guy's ``strong law of small numbers'' being in action.

\begin{question}[Distribution of Benford Errors]
\label{questionDE}
Which, if any, of the distributions observed in Figure \ref{fig:histograms}
are ``for real''  in the sense that they represent the true 
asymptotic behavior  
of the Benford errors?
\end{question}


\section{Unraveling the Digit 1 and 4 Mysteries, I.}
\label{sec:mystery1}

In this section we focus on the sequence $\{2^n\}$.  Using entirely
elementary arguments, we unravel some of the mysteries surrounding
the leading digit behavior of this sequence we have observed above.

We write
\[
S_d(N)=S_d(N,\{2^n\}),\quad
E_d(N)=E_d(N,\{2^n\})
\]
for the leading digit counting functions, respectively,  the Benford error
functions, associated with the sequence $\{2^n\}$.  We will need a
slight generalization of $S_d(N)$, defined by 
\begin{align}
\label{eq:def-SI}
S_I(N)=S_I(N,\{2^n\})&=\#\{n\le N: D(2^n)\in I\},
\end{align}
where $I$ is an interval in $[1,10)$.

The key to unlocking the digit $1$ and $4$ mysteries for the sequence
$\{2^n\}$ is contained in the following lemma, which provides an explicit
formula for $S_I(N)$ for certain intervals $I$. 

\begin{lem}
\label{lem:key-lemma}
Let $N\in\NN$ and $d\in\{1,2,\dots,5\}$. Then
\begin{equation}
\label{eq:key-lemma}
S_{[d,2d)}(N)=
\begin{cases}
\fl{N\lgt 2} & \text{if $d=1$,}
\\[.5ex]
\fl{N\lgt 2+\lgt(10/d)} & \text{if $2\le d\le 5$.}
\end{cases}
\end{equation}
\end{lem}

\begin{proof}
Let $N\in\NN$ and $d\in\{1,2,\dots,5\}$ be given.

Suppose first that $d>1$ and $2^N<d$. In this case we have 
$2^n\le 2^N<d$ and hence $D(2^n)<d$ for all $n\le N$, 
and thus $S_{[d,2d)}(N)=0$.
On the other hand, in view of the inequalities
\[
0< N\lgt2+\lgt(10/d)=\lgt(2^N/d)+1<1,
\]
we have $\fl{N\lgt2+\lgt(10/d)}=0$.  Therefore \eqref{eq:key-lemma} holds
trivially when $2^N<d$, and we can henceforth assume that 
\begin{equation}
\label{eq:key-lemma-hyp}
2^N\ge d.
\end{equation}

Let $k$ be the unique integer satisfying
\begin{equation}
\label{eq:k-def}
d\cdot 10^k \le 2^N<d\cdot 10^{k+1}.
\end{equation}
Our assumption \eqref{eq:key-lemma-hyp} ensures that $k$ is a
\emph{nonnegative} integer, and rewriting \eqref{eq:k-def} as
\[
\lgt d +k\le N \lgt2< \lgt d + k+1
\]
yields the explicit formula 
\begin{equation}
\label{eq:k-def2}
k=\Fl{N\lgt 2 - \lgt d}.
\end{equation}

Now observe that $d\le D(2^n)<2d$ holds if and only if 
$2^n$ falls into one of the intervals
\begin{equation}
\label{eq:key-lemma-d}
[d\cdot 10^i,2d\cdot 10^i), \quad i=0,1,\dots .
\end{equation}
Since each such interval is of the form $[x,2x)$, 
it contains exactly one term $2^n$.
Hence, since $2^N$ is one such term, 
the number of integers $n\le N$ counted in $S_{[d,2d)}(N)$ 
is equal to the number of integers $i$ for which 
the interval \eqref{eq:key-lemma-d} overlaps with the range $[2^1,2^N]$.
By the definition of $k$ (see \eqref{eq:k-def}), this holds if and only if
$1\le i\le k$ in the case $d=1$, and if and only if $0\le i\le k$ in the
case $d\ge 2$. Thus, 
$S_{[d,2d)}(N)$ 
is equal to $k$ in the first case, and $k+1$ in the second case. 
Substituting the explicit formula \eqref{eq:k-def2} for $k$
then yields the desired relation \eqref{eq:key-lemma}.
\end{proof}

From Lemma \ref{lem:key-lemma} we derive our first main result,  
an explicit formula for the Benford errors $E_1(N)$ and $E_4(N)$
associated with the sequence $\{2^n\}$.

\begin{thm}[Digit 1 and 4 Benford Errors for $\{2^n\}$]
\label{thm:digit1-4-errors}
Let $N$ be a positive integer. Then the Benford errors $E_d(N)=
E_d(N,\{2^n\})$ satisfy
\begin{align}
\label{eq:d1-explicit-formula}
E_1(N)&=-\Fr{N\alpha},
\\
\label{eq:d4-explicit-formula}
E_4(N)&=
\Fr{N\alpha}
+\Fr{N\alpha -\alpha}
+\Fr{N\alpha +\alpha} -1,
\end{align}
where $\alpha=\lgt2$.
(Recall that $\{x\}$ denotes the fractional part of $x$.)
\end{thm}

\begin{proof}
Since $S_1(N)=S_{[1,2)}(N)$, 
applying Lemma \ref{lem:key-lemma} with $d=1$ gives
\begin{align*}
E_1(N)&=S_1(N)-B_1(N) 
=\fl{N\lgt 2}-N\lgt\left(1+\frac11\right) =- \Fr{N\lgt 2},
\end{align*}
which proves \eqref{eq:d1-explicit-formula}.

The proof of the second formula is slightly more involved. Noting that 
$S_4(N)=S_{[4,5)}(N)$, we have
\begin{equation*}
S_4(N)= N-S_{[1,2)}(N) -S_{[2,4)}(N) -S_{[5,10)}(N).
\end{equation*}
Applying Lemma \ref{lem:key-lemma} to each of the terms on the right of this
relation, we obtain 
\begin{align*}
S_4(N)&=N-
\Fl{N\lgt 2}
-\Fl{N\lgt 2 + \lgt \frac{10}{2}}
-\Fl{N\lgt 2 + \lgt \frac{10}{5}}
\\
&=N\left(1-3\lgt2\right) -1 + \Fr{N\alpha}
-\Fr{N\alpha + 1-\alpha} -\Fr{N\alpha  + \alpha}
\\
&=N\lgt\frac{5}{4}
-1 + \Fr{N\alpha}
-\Fr{N\alpha -\alpha} -\Fr{N\alpha  + \alpha}.
\end{align*}
Since $E_4(N)=S_4(N)-N\lgt(5/4)$, this yields the desired formula 
\eqref{eq:d4-explicit-formula}.
\end{proof}

As an immediate consequence of the formulas 
\eqref{eq:d1-explicit-formula} and \eqref{eq:d4-explicit-formula} we
obtain the bounds
\begin{align}
\label{eq:d1-bound}
&-1<E_1(N)\le 0
\quad \text{for all $N\in\NN$},
\\
\label{eq:d4-bound}
&-1\le E_4(N)<2\quad \text{for all $N\in\NN$}.
\end{align}
In particular, the Benford errors for digits $1$ and $4$ for the sequence
$\{2^n\}$ are bounded, thus providing a partial answer to Question
\ref{questionBE}.  
Moreover, the bound \eqref{eq:d1-bound}  is precisely the condition
\eqref{eq:perfect-hit1-criterion} characterizing a (lower) perfect hit,  
so the Benford prediction for leading digit $1$ for the sequence
$\{2^n\}$ is indeed a true perfect hit in the  sense of Definition
\ref{def:perfect-hits}.  Hence, at least one of the nine (potential)
``perfect hits'' observed in Table \ref{table:benford-1bil} turned out
to be ``for real'': the leading digit $1$ count for the sequence
$\{2^n\}$ is \emph{always} equal to the Benford prediction rounded  
down to the nearest integer.

What about the other eight entries in this table that represented
perfect hits at $N=10^9$?  Are these ``for real'' as well, or are they
mere coincidences?  Are there cases where rounding \emph{up} the Benford
prediction always gives the exact leading digit count?
We will address these questions in Section \ref{sec:bounded-error} below,
but we first use the results of Theorem \ref{thm:digit1-4-errors}
to settle another numerical mystery, namely the distribution of the
Benford errors for digits $1$ and $4$ in Figure \ref{fig:histograms}.

\section{Unraveling the Digit 1 and 4 Mysteries, II.}
\label{sec:mystery2}

Continuing our focus on the sequence $\{2^n\}$, we now turn to the  
\emph{distribution} of the Benford errors $E_1(N)$ and $E_4(N)$ for this
sequence and we seek to  explain the peculiar shapes of these
distributions that we had observed in Figure \ref{fig:histograms}.
We will prove the following. 

\begin{thm}[Distribution of Digit 1 and 4 Benford Errors for $\{2^n\}$]
\label{thm:digit1-4-distribution}
The sequences $\{E_1(n)\}$ and $\{E_4(n)\}$ satisfy, 
for any real numbers $s<t$, 
\begin{equation}
\label{eq:limit-distribution}
\lim_{N\to\infty}\frac1N\#\{n\le N: s\le E_i(n)<t\}=\int_s^tf_i(x)\,dx
\quad (i=1,4),
\end{equation}
where $f_1(x)$ and $f_4(x)$ are defined by
\begin{align}
\label{eq:d1-density}
f_1(x)&=\begin{cases} 1 &\text{if\, \ $-1\le x\le 0$,}
\\
0&\text{otherwise,}
\end{cases}
\\
\label{eq:d4-density}
f_4(x)&=\begin{cases} 
1/3 & \text{if\, \ $3\alpha-1\le x\le 0$\,\  or\, \ $1\le x<2-3\alpha$,}
\\
2/3 & \text{if\, \ $0\le x<1-3\alpha$\,\ or\, \ $3\alpha \le x<1$,}
\\
1 & \text{if\, \ $1-3\alpha\le x<3\alpha$,}
\\
0 & \text{otherwise,}
\end{cases}
\end{align}
where $\alpha=\lgt 2=0.30103\dots$.  
\end{thm}

The function $f_1(x)$ is the probability density of a uniform distribution
on the interval $[-1,0]$.  The function $f_4(x)$,
shown in Figure \ref{fig:f4} below, is a weighted average of three uniform
densities supported on the intervals 
$[1-3\alpha,3\alpha]\approx [0.097,0.903]$, 
$[0,1]$, and $[3\alpha-1,2-3\alpha]\approx [-0.097,1.097]$, respectively. 

\begin{figure}[H]
\begin{center}
\includegraphics[width=.7\textwidth,height=0.35\textwidth]
{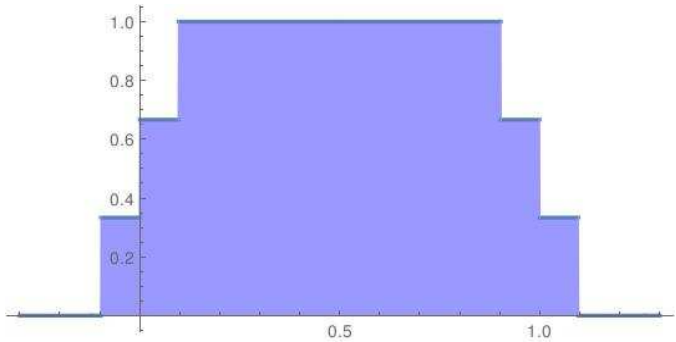}
\end{center}
\caption{The probability density function $f_4(x)$.}
\label{fig:f4}
\end{figure}

The theorem shows that the error distributions for digits $1$ and
$4$ we had observed in Figure \ref{fig:histograms} are ``for real'':
The digit $1$ error is indeed uniformly distributed over the interval
$[-1,0]$, while the ``mystery distribution'' of the digit $4$ error 
turns out to be a superposition of three uniform distributions,
given by the density function $f_4(x)$. 

\begin{proof}[Proof of Theorem \ref{thm:digit1-4-distribution}]
By Theorem \ref{thm:digit1-4-errors} we have 
\begin{align*}
E_1(n)&=-\fr{n\alpha},
\\
E_4(n)&=\fr{n\alpha}+\fr{n\alpha+\alpha} +\fr{n\alpha-\alpha}-1.
\end{align*}
The distribution of the numbers $\{n\alpha\}$ in these formulas 
is well understood: Indeed, since $\alpha=\lgt2$ is irrational, by  
\emph{Weyl's theorem} (see, e.g.,  \cite{weyl}), these numbers 
behave like a uniform random variable on the interval $[0,1]$,
in the sense that 
for any real numbers $s,t$ with $0\le s<t\le 1$,  
\begin{equation}
\notag
\label{eq:weyl}
\lim_{N\to\infty}\frac1N \#\{n\le N:  s\le \{n\alpha\}<t\}=t-s.
\end{equation}
It follows that 
the limit distributions of $E_1(n)$ and $E_4(n)$ exist and 
are those of the random variables 
\begin{align}
\label{eq:X1}
X_1&=-U,
\\
\label{eq:X4}
X_4&= U+\fr{U+\alpha}+\fr{U-\alpha}-1,
\end{align}
where $U$ is a uniform random variable on $[0,1]$.

From \eqref{eq:X1} we immediately obtain that $X_1$ is a
uniform random variable on $[-1,0]$ and hence has density given by the
function $f_1(x)$ defined above. Moreover, 
by considering separately the ranges $0\le U<\alpha$, $\alpha\le U<1-\alpha$, 
and $1-\alpha\le U\le 1$ in \eqref{eq:X4}, one can check that
$X_4$ is a superposition of three uniform distributions 
corresponding to these three ranges, and that the  
density of $X_4$ is given by the function $f_4(x)$ defined above;
we omit the details of this routine, but somewhat tedious, calculation.
\end{proof}


\section{Benford Errors and Interval Discrepancy.}
\label{sec:benford-interval-discrepancy}

We now consider the case of a general geometric sequence $\{a^n\}$, 
where $a$ is a positive real number (not necessarily an integer),
subject only to the condition
\begin{equation}
\label{eq:a-condition}
\lgt a\not\in \QQ.
\end{equation}
Condition \eqref{eq:a-condition} serves to exclude sequences such as
$\{\sqrt{10}^{\,\,n}\}$ for which the leading digits behave in a
trivial manner.

To make further progress, we exploit the connection between the
distribution of leading digits of a sequence and the theory of uniform
distribution modulo $1$.  This connection is well known, and it has been
used to rigorously establish Benford's law for various classes of
mathematical sequences; see, for example, \cite{diaconis1977}.  For
our purposes, we need a specific form of this connection that involves the 
concept of \emph{interval discrepancy} defined as follows. 

\begin{defn}[Interval Discrepancy] 
\label{defn:interval-discrepancy}
Let $\alpha$ be a real number,  and let $I$ be an interval in
$[0,1]$.  For any $N\in\NN$, we define 
the \textbf{interval discrepancy} of the sequence $\{n\alpha\}$ with
respect to the interval $I$  by 
\begin{equation}
\label{eq:def-interval-discrepancy}
\Delta(N,\alpha,I)=\#\{n\le N: \{n\alpha\}\in I\}- N|I|,
\end{equation}
where $|I|$ denotes the length of $I$.  
\end{defn}

The point of this definition is that it allows us to express the Benford
error $E_d(N,\{a^n\})$ directly in the form $\Delta(N,\alpha,I)$ with   
suitable choices of $\alpha$ and $I$: 

\begin{lem}[Benford Errors and Interval Discrepancy]
\label{lem:benford-interval-discrepancy}
Let $a$ be a positive real number, $N\in\NN$, and $d\in\{1,2,\dots,9\}$.
Then we have
\begin{equation}
\label{eq:benford-interval-discrepancy}
E_d(N,\{a^n\})=\Delta(N,\alpha,[\lgt d, \lgt (d+1)),
\end{equation}
where $\alpha=\lgt a$.
\end{lem}

\begin{proof} 
Note that, for any $n\in\NN$, 
\begin{align*}
D(a^n)=d &\Longleftrightarrow d\cdot 10^i\le a^n < (d+1)\cdot 10^{i}
\quad\text{for some $i\in\ZZ$}
\\
&\Longleftrightarrow \lgt d + i \le n\lgt a < \lgt(d+1)+ i
\quad\text{for some $i\in\ZZ$}
\\
&\Longleftrightarrow \fr{n\alpha} \in [\lgt d,\lgt (d+1)),
\end{align*}
since $\lgt a^n = n\lgt a = n\alpha$.  It follows that 
\begin{equation*}
S_d(N,\{a^n\})=\#\{n\le N: \fr{n\alpha}\in[\lgt d,\lgt(d+1))\},
\end{equation*}
and subtracting $B_d(N)=N\lgt(1+1/d)=N(\lgt(d+1)-\lgt d)$ on each side 
yields the desired relation \eqref{eq:benford-interval-discrepancy}.
\end{proof}

We remark that the interval discrepancy defined above 
is different from the usual notion of discrepancy of a sequence
in the theory of uniform distribution modulo $1$,  defined as 
(see, for example, \cite{drmota1997,kuipers1974})
\begin{equation*}
\label{eq:def-discrepancy}
D_N(\{n\alpha\})=\frac1N\sup_{0\le s<t\le 1} |\Delta(N, \alpha,[s,t))|.
\end{equation*}
While there exists a large body of work on the asymptotic behavior 
of the ordinary discrepancy function $D_N$, much less is known about 
the interval discrepancy $\Delta(N,\alpha, I)$.  In the following 
sections we describe some of the key results on interval discrepancies, and
we apply these results to Benford errors.  


\section{Interval Discrepancy: Results of Ostrowski 
and Kesten.}
\label{sec:interval-discrepancy1}

In view of Lemma 
\ref{lem:benford-interval-discrepancy}, the question of whether the Benford
error is bounded leads naturally to the 
following question about the  behavior of the interval discrepancy:

\begin{question}
Under what conditions on $\alpha$ and $I$ is 
the interval discrepancy $\Delta(N,\alpha,I)$ bounded as $N\to\infty$?
\end{question}

It turns out that this question has a simple and elegant answer, given by 
the following theorem of Kesten.

\begin{prop}[Bounded Interval Discrepancy \cite{kesten}]
\label{prop:kesten}
Let $\alpha$ be irrational, and let $I=[s,t)$, where $0\le s<t\le 1$.
Then $\Delta(N,\alpha,I)$ is bounded as $N\to\infty$ if and only if 
\begin{equation}
\label{eq:kesten-condition}
t-s=\{k\alpha\}
\quad \text{for some $k\in\ZZ\setminus\{0\}$.}
\end{equation}
\end{prop}

This result has an interesting history going back nearly a century. 
The sufficiency of condition \eqref{eq:kesten-condition}  
was established by Hecke \cite{hecke1922} in 1922 for the special case $s=0$
and by Ostrowski \cite{ostrowski1927} in 1927 for general $s$.
The necessity of the condition had been conjectured by
Erd\H os and Sz\" usz \cite{erdos} and was proved by Kesten \cite{kesten} in 1966.

For the case when condition \eqref{eq:kesten-condition} is satisfied, we
have the following  more precise result that gives  an explicit formula
for the interval discrepancy. This result is implicit in Ostrowski's paper
\cite{ostrowski1930} (see formulas (6) and (6') in 
\cite{ostrowski1930}), but since the original paper is not easily accessible, 
we will provide a proof here. 

\begin{prop}[Explicit Formula for Interval Discrepancy 
\cite{ostrowski1930}]
\label{prop:ostrowski}
Let $\alpha$ be irrational, $k\in\ZZ\setminus\{0\}$, and 
$0\le s\le 1-\fr{k\alpha}$. Then we have, for any $N\in\NN$, 
\begin{align}
\label{eq:ostrowski}
\Delta(N,\alpha,[s,s+\{k\alpha\})) &
=
\begin{cases}
\displaystyle
-\sum_{h=0}^{k-1}\Bigl(\fr{N\alpha-h\alpha-s} - \fr{-h\alpha-s}\Bigr)
&\text{if $k>0$,}
\\
\displaystyle
\sum_{h=1}^{|k|}\Bigl(\fr{N\alpha+ h\alpha-s} - \fr{h\alpha-s}\Bigr)
&\text{if $k<0$.}
\end{cases}
\end{align}
\end{prop}

\begin{proof}
Let $\alpha$, $k$, and $s$ be given as in the proposition.
We start with the elementary identity
\begin{equation}
\label{eq:fractional-part-identity}
\fr{x-t}-\fr{x-s}=
\begin{cases} 1-(t-s) & \text{if $s\le \{x\}<t$,}
\\
-(t-s) & \text{otherwise.}
\end{cases}
\end{equation}
which holds for any real numbers $x$ and $t$ with $0\le s<t\le 1$.
Setting $x=\{n\alpha\}$ in \eqref{eq:fractional-part-identity}
and summing over  $n\le N$, we obtain
\begin{align*}
\sum_{n=1}^N\Bigl(\fr{\fr{n\alpha}-t}-\fr{\fr{n\alpha}-s}\Bigr)
&=\#\{n\le N: \{n\alpha\}\in [s,t)\}- N(t-s)
\\
&=\Delta(N,\alpha,[s,t)).
\end{align*}
Specializing $t$ to $t=s+\{k\alpha\}$, the latter sum turns into 
a telescoping sum in which all except the first and last $|k|$ terms
cancel out.  More precisely, if $k>0$, then
\begin{align*}
\Delta(N,\alpha,[s,s+\{k\alpha\}))
&=\sum_{n=1}^N\Bigl(\fr{\fr{n\alpha}-s-\{k\alpha\}}-\fr{\fr{n\alpha}-s}\Bigr)
\\
&=\sum_{n=1}^N\Bigl(\fr{(n-k)\alpha-s}-\fr{n\alpha-s}\Bigr)
\\
&=-\sum_{h=0}^{k-1}\Bigl(\fr{N\alpha-h\alpha-s}
-\{-h\alpha-s\}\Bigr),
\end{align*}
which proves the first case of \eqref{eq:ostrowski}.
The second case follows by an analogous argument. 
\end{proof}

\section{Perfect Hits and Bounded Errors: The General Case.}
\label{sec:bounded-error}

With the theorems of Kesten and Ostrowski at our disposal, we are
finally in a position to settle Questions \ref{questionPH} and
\ref{questionBE}  and provide a partial
answer to Question \ref{questionDE}.  Our main result is the following theorem, which
gives a complete description of all (nontrivial) geometric sequences
$\{a^n\}$ and digits $d\in\{1,2,\dots,9\}$ for which the Benford
prediction has bounded error or represents a ``perfect hit''
in the sense of Definition \ref{def:perfect-hits}.

\begin{thm}[Perfect Hits and Bounded Benford Errors]
\label{thm:bounded-error}
Let $a$ be a positive real number satisfying \eqref{eq:a-condition}, and
let $d\in\{1,2,\dots,9\}$.
\begin{enumerate}
\item[(i)] \textbf{Characterization of bounded Benford errors.}
The Benford prediction for the leading digit $d$ in $\{a^n\}$ 
has \textbf{bounded error} if and only if 
\begin{equation}
\label{eq:bounded-error-condition}
a^k=\frac{d+1}d 10^m \quad \text{for some $k\in\ZZ\setminus\{0\}$ and $m\in\ZZ$.}
\end{equation}
Moreover, if this condition is satisfied, then the limit distribution of
the Benford error $E_d(N,\{a^n\})$ (in the sense of
\eqref{eq:limit-distribution}) exists and is a weighted average of
$|k|$ uniform distributions.

\item[(ii)] \textbf{Characterizations of perfect hits.}
The Benford prediction for the leading digit $d$ in $\{a^n\}$ is
\begin{itemize}
\item[$\bullet$] 
a \textbf{lower perfect hit} 
(i.e., satisfies $S_d(N,\{a^n\})=\fl{B_d(N)}$ for all $N\in\NN$)
if and only if  
\begin{equation}
\label{eq:perfect-hit1-condition}
d=1 \ \text{and}\ a=2\cdot 10^m \quad \text{for some $m\in\ZZ$;}
\end{equation}
\item[$\bullet$] 
an \textbf{upper perfect hit} 
(i.e., satisfies $S_d(N,\{a^n\})=\cl{B_d(N)}$ for all $N\in\NN$)
if and only if  
\begin{equation}
\label{eq:perfect-hit2-condition}
d=9 \ \text{and}\ a=9\cdot 10^m \quad \text{for some $m\in\ZZ$.}
\end{equation}
\end{itemize}
Moreover, if one of the conditions \eqref{eq:perfect-hit1-condition} and
\eqref{eq:perfect-hit2-condition} holds, then the limit distribution of
the Benford error $E_d(N,\{a^n\})$ exists and is uniform on $[-1,0]$ in
the first case, and uniform on $[0,1]$ in the second case.
\end{enumerate}
\end{thm}

Theorem \ref{thm:bounded-error} can be viewed as a far-reaching
generalization of the results (Theorems \ref{thm:digit1-4-errors} and
\ref{thm:digit1-4-distribution}) we had obtained above, using a much more
elementary approach, for the cases of leading digits $1$ and $4$ in the
sequence $\{2^n\}$.  In particular, the theorem shows that the 
``brick-shaped'' error distributions we had observed in these particular
cases are ``for real,'' and that error distributions of this type arise 
whenever the sequence $\{a^n\}$ and the digit $d$ 
satisfy the boundedness criterion \eqref{eq:bounded-error-condition} of
Theorem \ref{thm:bounded-error}. 

\subsection{Special cases and consequences.}
If $a$ is an integer $\ge2$ that is not divisible by $10$, then condition
\eqref{eq:bounded-error-condition} reduces to a simple Diophantine equation for
the number $a$ and the digit $d$. This equation has only finitely many
solutions, which can be found by considering the prime factorizations of
the numbers $a$, $d$, and $d+1$. Table \ref{table:bounded-errors} gives a
complete list of these solutions. 

%
\begin{table}[H]
\caption{%
Complete 
list of digits $d$ and sequences $\{a^n\}$, where $a\ge2$ is an
integer not divisible by $10$, for which the Benford prediction has
bounded error. The two entries in \exact{boldface}, corresponding to 
the digit-sequence pairs 
$(1,\{2^n\})$ and $(9,\{9^n\})$,
denote cases where the Benford prediction is a perfect hit 
in the sense of Definition \ref{def:perfect-hits}. 
}
\label{table:bounded-errors} 
\begin{center}
\begin{tabular}{|c||l|}
\hline
Digit $d$ & Sequences $\{a^n\}$ with bounded Benford error  
\\
\hline
\hline
1 & 
 \mathexact{$\mathbf{\{2^n\}}$},
 ${\{5^n\}}$
\\
\hline
2 &
 ${\{15^n\}}$
\\
\hline
3 & 
 ${\{75^n\}}$
\\
\hline
4 & 
 $\{2^n\}$, $\{5^n\}$, 
 $\{8^n\}$, $\{125^n\}$
\\
\hline
5 & 
 ${\{12^n\}}$
\\
\hline
6 & 
\\
\hline
7 & 
 ${\{875^n\}}$
\\
\hline
8 &
 ${\{1125^n\}}$
\\
\hline
9 & 
$\{3^n\}$,  \exact{$\mathbf{\{9^n\}}$}
\\
\hline
\end{tabular}
\end{center}
\end{table}

Theorem \ref{thm:bounded-error}, in conjunction with Table
\ref{table:bounded-errors},
allows us to completely settle Questions \ref{questionPH} and
\ref{questionBE} on the true nature of the (potential) ``perfect
hits'' observed  in Table \ref{table:benford-1bil}: Of the 27 entries in
this table, only the digit-sequence pair $(1,\{2^n\})$ satisfies the
perfect hit criterion of Theorem \ref{thm:bounded-error} and thus
represents a true perfect hit (more precisely, a lower perfect hit).
An additional five entries---the pairs
$(1,\{5^n\})$, $(4,\{2^n\})$, $(4,\{5^n\})$, and $(9,\{3^n\})$---appear in
Table \ref{table:bounded-errors} and thus represent cases in which the
Benford error is bounded. On the other hand, none of the remaining 21
entries in Table \ref{table:benford-1bil} appears in Table
\ref{table:bounded-errors}, so in all of these cases the Benford error is
unbounded and the ``perfect hits'' observed in Table
\ref{table:benford-1bil} for some of these cases are mere coincidences.

\begin{proof}[Proof of Theorem \ref{thm:bounded-error}]
(i) Let $a$ and $d$ be given as in the theorem, and set 
\begin{equation}
\label{eq:def-alpha-s}
\alpha=\lgt a,\ s=\lgt d,\ t=\lgt (d+1).
\end{equation}
With these notations 
the condition \eqref{eq:bounded-error-condition}
can be restated as follows.
\begin{align}
\label{eq:bounded-error-condition-alt}
&t=s+\{k\alpha\} \quad \text{for some $k\in\ZZ\setminus\{0\}$.}
\end{align}
We begin by showing that
\eqref{eq:bounded-error-condition-alt}
holds if and
only if the Benford error is bounded.
By Lemma \ref{lem:benford-interval-discrepancy} we have
\begin{equation}
\label{eq:main-proof1}
E_d(N,\{a^n\})= \Delta(N,\alpha,[s,t)).
\end{equation}
By Kesten's theorem (Proposition~\ref{prop:kesten}) 
and since by assumption $\alpha$ is irrational,
it follows that $E_d(N,\{a^n\})$ is
bounded as a function of $N$ if and only if condition
\eqref{eq:bounded-error-condition-alt} holds. 
This proves the first assertion of part (i) of the theorem. 

Next, we assume that \eqref{eq:bounded-error-condition-alt} holds and 
consider the distribution of the Benford error in this case.  
Combining \eqref{eq:main-proof1} with Ostrowski's theorem 
(Proposition~\ref{prop:ostrowski}) 
gives the explicit formula
\begin{align}
\label{eq:main-proof2}
E_d(N,\{a^n\})
=
\begin{cases}
\displaystyle
-\sum_{h=0}^{k-1}\Bigl(\fr{N\alpha-h\alpha-s} - \fr{-h\alpha-s}\Bigr)
&\text{if $k>0$,}
\\
\displaystyle
\sum_{h=1}^{|k|}\Bigl(\fr{N\alpha+ h\alpha-s} - \fr{h\alpha-s}\Bigr)
&\text{if $k<0$,}
\end{cases}
\end{align}
with $s=\lgt d$ and $\alpha=\lgt a$ as in \eqref{eq:def-alpha-s}.
Since, by Weyl's theorem (see \eqref{eq:weyl}), the 
sequence $\{N\alpha\}$ is uniformly distributed modulo
$1$, it follows that the 
Benford errors $E_d(N,\{a^n\})$ have a distribution equal to that 
of the random variable 
\begin{equation}
\label{eq:main-proof3}
X_{k,\alpha,s}=
\begin{cases} \displaystyle
-\sum_{h=0}^{k-1}\Bigl(\fr{U-h\alpha-s} - \fr{-h\alpha-s}\Bigr)
&\text{if $k>0$,}
\\
\displaystyle
\sum_{h=1}^{|k|}\Bigl(\fr{U+ h\alpha-s} - \fr{h\alpha-s}\Bigr)
&\text{if $k<0$,}
\end{cases}
\end{equation}
where $U$ is uniformly distributed on $[0,1]$.
The latter distribution is clearly a superposition of $|k|$ uniform
distributions, thus proving the last assertion of part (i) of the
theorem.

\medskip

(ii) If \eqref{eq:perfect-hit1-condition} or
\eqref{eq:perfect-hit2-condition} holds,  
then the bounded error condition \eqref{eq:bounded-error-condition} is
satisfied with $k=1$ or $k=-1$, so the two expressions for 
the Benford error $E_d(N,\{a^n\})$ in \eqref{eq:main-proof2} 
reduce to a single term, and it is easily checked that this term is equal
to $-\{N\alpha\}$ in case \eqref{eq:perfect-hit1-condition} holds, 
and $\{N\alpha\}$ if \eqref{eq:perfect-hit2-condition} holds. 
In the first case, $E_d(N,\{a^n\})$ is contained in the interval $(-1,0)$ and 
thus satisfies the lower perfect hit criterion \eqref{eq:perfect-hit1-criterion},
while in the second case $E_d(N,\{a^n\})$ falls into the interval
$(0,1)$ and satisfies the upper perfect hit criterion
\eqref{eq:perfect-hit2-criterion}.  In either case, Weyl's theorem shows
that $E_d(N,\{a^n\})$ is uniformly distributed over the respective
interval.

For the converse direction, let $a$ and $d$ be given as in the theorem.
and suppose that the Benford prediction for digit $d$ and the sequence
$\{a^n\}$ is a (lower or upper) perfect hit.  In particular, this implies
that the Benford error, $E_d(N,\{a^n\})$, is bounded, so by part (i) 
$E_d(N,\{a^n\})$ has a limit distribution given by the random variable 
$X_{k,\alpha,s}$ defined in \eqref{eq:main-proof3}, where
$k$ is a a nonzero integer and $\alpha$ and $s$ are given by
\eqref{eq:def-alpha-s}.

Suppose first that $k=\pm1$. Then \eqref{eq:main-proof3} reduces
to
\begin{equation*}
X_{k,\alpha,s}=\begin{cases} 
-\fr{U-s}+\fr{-s}
&\text{if $k=1$,}
\\
\fr{U +\alpha-s} - \fr{\alpha-s}
&\text{if $k=-1$.}
\end{cases}
\end{equation*}
Thus, $X_{k,\alpha,s}$ has uniform distribution on the interval
$[-1+\theta,\theta]$, where  $\theta=\{-s\}$ if $k=1$, and
$\theta=1-\{\alpha-s\}$ if $k=-1$. Since we assumed that the Benford
prediction is a perfect hit, we must have either $\theta=0$ (for
a lower perfect hit) or $\theta=1$ (for an upper perfect hit). In the first
case, we have $k=1$ and $s=0$, and hence $d=1$, $t=\log_{10}2$, and
$\fr{\alpha}=t-s=\log_{10}2$. Therefore $a=10^\alpha=2\cdot 10^m$ for some
integer $m$, which is the desired condition \eqref{eq:perfect-hit1-condition} 
for a lower perfect hit.  A similar argument shows that in the case
$\theta=1$, the upper perfect hit condition, \eqref{eq:perfect-hit2-condition},
holds.

Now suppose $|k|\ge 2$.  To complete the proof of the necessity of the
conditions \eqref{eq:perfect-hit1-condition} and
\eqref{eq:perfect-hit2-condition}, it suffices to show that in this case
we cannot have a perfect hit.  We will do so by showing that the support of
the random variable $X_{k,\alpha,s}$ in \eqref{eq:main-proof3}
covers an interval of length greater than $1$ and thus, in particular,
cannot be equal to one of the intervals $[-1,0]$ and $[0,1]$ corresponding
to a perfect hit.

If we set $U'=\fr{U-s}$ if $k>0$ and $U'=\fr{U+|k|\alpha-s}$ if $k<0$, 
then $U'$ is uniformly distributed on $[0,1]$, 
and \eqref{eq:main-proof3} can be written in the form 
\begin{equation} 
\label{eq:main-proof4}
X_{k,\alpha,s}=U'+\sum_{h=1}^{|k|-1} \{U'-\fr{h\alpha}\} + C, 
\end{equation} 
where $C=C(k,\alpha,s)$ is a constant. (Note that since $|k|\ge 2$, the sum
on the right of \eqref{eq:main-proof4} contains at least one term.)
Now let 
$0<\lambda_1<\dots < \lambda_{|k|-1}<1$ denote 
the numbers $\{h\alpha\}$, $h=1,\dots,|k|-1$, arranged in increasing
order, and set $\lambda_0=0$ and $\lambda_{|k|}=1$.  Then
\eqref{eq:main-proof4} yields 
\begin{equation*} 
X_{k,\alpha,s}=|k|U'  + C_i
\quad \text{if $\lambda_i\le U'<\lambda_{i+1}$}
\end{equation*}
for each $i\in\{0,1,\dots,|k|-1\}$,  
where $C_i=C_i(k,\alpha,s)$ is a constant. 
In particular, for each such $i$ the support of 
$X_{k,\alpha,s}$ covers an interval of length
$|k|(\lambda_{i+1}-\lambda_i)$. so we have
\begin{equation}
\label{eq:main-proof6}
\max X_{k,\alpha,s} - \min X_{k,\alpha,s}\ge 
|k|(\lambda_{i+1}-\lambda_i).
\end{equation}
By the pigeonhole principle, one of the intervals
$[\lambda_i,\lambda_{i+1})$, $i=0,\dots,|k|-1$, 
must have length $>1/|k|$
except in the case when $\lambda_i=i/|k|$ for $i=0,1,\dots,|k|$.
But this case  is impossible since the numbers $\lambda_i$  are a
permutation of numbers of the form $\{h\alpha\}$, $h=0,\dots,|k|-1$, 
and $\alpha$ is
irrational.  It follows that, for some $i\in\{0,\dots,|k|-1\}$, the
right-hand side of \eqref{eq:main-proof6} is strictly greater than $1$.
Hence $X_{k,\alpha,s}$ is supported on an interval of length greater than
$1$, and the proof of Theorem \ref{thm:bounded-error} is complete.
\end{proof}


\section{The Final Frontier: The Case of Unbounded Errors.}
\label{sec:unbounded-error}

Having characterized the cases when the Benford error is bounded and
completely described the behavior of the Benford error for those cases, we
now turn to the final---and deepest---piece of the puzzle, the behavior of
the Benford error in cases where it is unbounded, i.e., when the
boundedness criterion \eqref{eq:bounded-error-condition} of Theorem
\ref{thm:bounded-error} is not satisfied.

\subsection{Exhibit B, Revisited.}
For the sequence $\{2^n\}$ the Benford error is unbounded exactly for
the digits $d=2,3,5,6,7,8,9$ (see Table \ref{table:bounded-errors}).
Remarkably, those are precisely the digits for which the distribution of the
Benford error in Figure \ref{fig:histograms} has the distinctive shape of
a normal distribution.  
Is this observed behavior for the sequence $\{2^n\}$ ``for real,'' in the
sense that the Benford error satisfies an appropriate central limit
theorem for these seven digits?
Is this behavior ``typical'' for cases of sequences $\{a^n\}$
and digits $d$ in which the Benford error is unbounded?  Could it be that
a central limit theorem holds in \emph{all}  cases in which the Benford
error is unbounded?  In other words, is it possible that the 
distribution of the Benford error for sequences $\{a^n\}$ is
either asymptotically normal, or a mixture of uniform distributions?

These are all natural questions suggested by numerical data, and it is 
not clear where the truth lies. Indeed, we do not \emph{know} the answer,
but we will provide heuristics \emph{suggesting} what the truth is and
formulate conjectures based on such heuristics.

\subsection{Interval Discrepancy, Revisited: The Limiting Distribution of
$\Delta(N,\alpha,I)$.}

In view of the connection between Benford errors and the interval discrepancy 
$\Delta(N,\alpha,I)$ (see Lemma \ref{lem:benford-interval-discrepancy}),
it is natural to consider analogous questions about the limiting
distribution of the interval discrepancy.  In particular, one can ask:

\begin{question}
Under what conditions on $\alpha$ and $I$ does  
the interval discrepancy $\Delta(N,\alpha,I)$ satisfy a central limit
theorem?
\end{question}

In contrast to the question about bounded interval discrepancy, which had
been completely answered more than 50 years ago by Ostrowski and Kesten
(see Propositions \ref{prop:kesten} and 
\ref{prop:ostrowski}),  the behavior of $\Delta(N,\alpha,I)$ 
in the case of unbounded interval discrepancy
turns out to be much deeper, and despite some
spectacular progress in recent years, a complete understanding remains
elusive.  

The recent progress on this question is largely due to J\'{o}zsef Beck,
who over the past three decades engaged in a systematic, and still
ongoing,  effort to attack questions of this type, for which Beck coined
the term ``probabilistic Diophantine approximation.'' 
Beck's work is groundbreaking and extraordinarily deep. The proofs
of the results cited below take up well over one hundred pages and draw on
methods from multiple fields, including algebraic and analytic number
theory, probability theory, Fourier analysis, and the theory of Markov chains.
Beck's recent book 
\cite{beck-book} provides a beautifully written, and exceptionally well
motivated, exposition of this work and the profound ideas that underlie
it.  We highly recommend this book to the reader interested in learning
more about this fascinating new field at the intersection of number theory
and probability theory.

Beck's main result on the behavior of $\Delta(N,\alpha,I)$ 
is the following theorem.  Detailed proofs can be found in his book
\cite{beck-book}, as well as in his earlier papers \cite{beck2010,beck2011}. 

\begin{prop}[Central Limit Theorem for Interval Discrepancy
(Beck {\cite[Theorem 1.1]{beck-book}})]
\label{prop:beck1}
Let $\alpha$ be a quadratic irrational and let $I=[0,s]$, where 
$s$ is a rational number in $[0,1]$.
Then $\Delta(N,\alpha, [0,s))$ satisfies the central limit theorem
\begin{align}
\notag
\lim_{N\to\infty}  
&\frac1N\#\left\{n\le N: 
u\le \frac{\Delta(N,\alpha,[0,s])-C_1\log N}{C_2\sqrt{\log N}} <  v\right\}
\\
&\qquad=\frac{1}{\sqrt{2\pi}}
\int_u^v e^{-x^2/2}\, dx
\quad\text{for all $u<v$,}
\label{eq:interval-discrepancy-CLT}
\end{align}
where $C_1=C_1(\alpha,s)$ and $C_2=C_2(\alpha,s)$ are constants depending on
$\alpha$ and $s$. 
\end{prop}

This result shows that, under appropriate conditions on $\alpha$ and $I$,
the interval discrepancy, $\Delta(N,\alpha,I)$, is approximately normally
distributed with mean and variance growing at a logarithmic rate.  This is
consistent with the behavior of the Benford error we had observed in
Figure \ref{fig:histograms} for the digits $2,3,5,6,7,8,9$.

Can Proposition \ref{prop:beck1} explain, and rigorously justify, these
observations?
Unfortunately, the assumptions on $\alpha$
and $I$ in the proposition are too restrictive to be applicable in
situations corresponding to Benford errors.  Indeed, by Lemma
\ref{lem:benford-interval-discrepancy},  the Benford error,
$E_d(N,\{a^n\})$, is equal to the interval discrepancy
$\Delta(N,\alpha,I_d)$ with $\alpha=\lgt a$ and $I_d=[\lgt d,\lgt (d+1))$.
However, Proposition \ref{prop:beck1} applies only to intervals with
\emph{rational endpoints} and thus does not cover intervals of the form
$I_d$. Moreover, in the cases of greatest interest such as the sequence
$\{2^n\}$, the number $\alpha=\lgt 2$ is not a quadratic irrational and
hence not covered by Proposition  \ref{prop:beck1}.

Of these two limitations to applying Proposition \ref{prop:beck1} to 
Benford errors, the restriction on the type of interval $I$ seems surmountable. 
Indeed, Beck \cite[p.~38]{beck2001} proved a 
central limit theorem similar  to \eqref{eq:interval-discrepancy-CLT} for
``random'' intervals $I$.  Hence, it is at least plausible that the result
remains valid for intervals of the type $I_d$ provided  $|I_d|$ is not of
the form $\{k\alpha\}$ for some $k\in\ZZ$, which, by Kesten's theorem
(Proposition~\ref{prop:kesten}), would imply bounded interval discrepancy.

\subsection{Beck's Heuristic.}
The restriction of $\alpha$ to
quadratic irrationals in  Proposition \ref{prop:beck1} is due to the fact
that quadratic irrationals have  a periodic continued fraction expansion,
which simplifies the argument.  Beck remarks that
this restriction can be significantly relaxed, and he provides a heuristic
for the class of numbers $\alpha$ for which a central limit theorem should
hold, which we now describe.

Consider the continued fraction expansion of $\alpha$:
\begin{equation}
\label{eq:continued-fraction}
\alpha= a_0+\dfrac{1}{a_1+\dfrac{1}{a_2 + \dfrac{1}{a_3 + \cdots }}}.
\end{equation}
Then, according to Beck's heuristic (see
\cite[p.~38]{beck2001}),
the interval discrepancy $\Delta(N,\alpha,I)$ behaves roughly like
\begin{equation}
\label{eq:beck-heuristic}
\Delta(N,\alpha,I)\approx \epsilon_1a_1+\epsilon_2a_2+\dots +
\epsilon_k a_k,
\end{equation} 
where the $\epsilon_i$ are independent random variables with values $\pm1$ 
and $k=k(N)$ is defined by $q_k\le N<q_{k+1}$, where 
$p_i/q_i$ denotes the $i$th partial quotient
in the continued fraction expansion of $\alpha$. 
By the standard central limit theorem in probability theory
(see, e.g., \cite[Section VIII.4]{feller-volume2}),
such a sum has an asymptotically normal distribution if 
it satisfies  
\begin{equation}
\label{eq:lindeberg}
\lim_{k\to\infty}\frac{a_k^2}{\sum_{i=1}^k a_i^2} = 0.
\end{equation}
Beck \cite[p.~247]{beck-book} concludes that a central limit theorem for
$\Delta(N,\alpha,I)$ can be expected to hold whenever $\alpha$ is an
irrational number whose continued fraction expansion satisfies
\eqref{eq:lindeberg}.  On the other hand, Beck \cite[p.~38]{beck2001} also
notes that in the above situation
\eqref{eq:lindeberg} is essentially necessary for
a central limit theorem to hold.

\subsection{Application to Benford Errors.}
Since, by Lemma
\ref{lem:benford-interval-discrepancy}, 
$E_d(N,\{a^n\})=\Delta(N,\alpha,I_d)$, where 
$\alpha=\lgt a$ and $I_d=[\lgt d,\lgt (d+1))$,
Beck's heuristic suggests the following conjecture. 

\begin{conj}[Central Limit Theorem for Benford Errors] 
\label{conj:benford-error-CLT}
Let $a>0$ be a real number satisfying \eqref{eq:a-condition}, and
suppose  
that the continued fraction expansion of  $\alpha=\lgt a$ satisfies  
\eqref{eq:lindeberg}. 
Then, for any digit $d\in\{1,2,\dots,9\}$ that does \emph{not} satisfy 
the ``bounded error'' condition \eqref{eq:bounded-error-condition} 
of Theorem \ref{thm:bounded-error}, the Benford error $E_d(N,\{a^n\})$ 
is asymptotically normally distributed in the sense that  
there exist sequences $\{A_N\}$ and $\{B_N\}$ such that
\begin{align}
\notag
\lim_{N\to\infty}  
&\frac1N\#\left\{n\le N: 
u\le \frac{E_d(N,\{a^n\})-A_N}{B_N} <  v\right\}
\\
&\qquad=\frac{1}{\sqrt{2\pi}}
\int_u^v e^{-x^2/2}\, dx
\quad\text{for all $u<v$.}
\label{eq:benford-error-CLT}
\end{align}
\end{conj}

This conjecture would explain the normal shape of the distributions  
of the Benford errors observed in Figure \ref{fig:histograms} 
\emph{if} the number $\alpha=\lgt 2$ has a continued fraction expansion
satisfying \eqref{eq:lindeberg}.  Unfortunately, we know virtually nothing
about the continued fraction expansion of $\lgt 2$ and thus are in no
position to determine whether or not $\lgt 2$ satisfies 
\eqref{eq:lindeberg}.  We are similarly ignorant 
about the nature of the continued fraction expansion of any 
number of the form
\begin{equation}
\label{eq:a-integer}
\alpha=\lgt a,\  a\in \NN,\  \lgt a\not\in\QQ.
\end{equation}
Thus Conjecture \ref{conj:benford-error-CLT} does not shed light on
the leading digit behavior of the simplest 
class of sequences $\{a^n\}$, namely those where $a$ a positive integer
that is not a power of $10$.

We can certainly construct numbers $a$ for which $\alpha=\lgt a$
satisfies \eqref{eq:lindeberg} (for example, 
numbers $a$ such that $\log_{10}a$ is a quadratic irrational), but those   
constructions are rather artificial, and they do not cover natural
families of numbers $a$ such as positive integers or rationals. 

If we are willing to believe that all numbers of the form 
\eqref{eq:a-integer} satisfy
\eqref{eq:lindeberg} \emph{and} assume the truth of Conjecture
\ref{conj:benford-error-CLT}, then  we would be able to conclude that 
the Benford error for sequences $\{a^n\}$ with $a$ as in
\eqref{eq:a-integer}
satisfies the dichotomy mentioned above: the error is either bounded with 
a limit distribution  that is a finite mixture of uniform
distributions, or unbounded with a normal limit distribution.
This would be a satisfactory conclusion to our original quest, but it
depends on a crucial assumption, namely that \eqref{eq:lindeberg} holds
for the numbers of the form \eqref{eq:a-integer}.  

How realistic is such an assumption?
Alas, it turns out that this assumption is not
at all realistic, in the sense that ``most'' real numbers $\alpha$ do
not satisfy \eqref{eq:lindeberg}.  Indeed, 
Beck \cite[p.~39]{beck2001} (see also \cite[p.~244]{beck-book})
showed that  
the Gauss--Kusmin theorem, a classical result
on the distribution of the terms $a_i=a_i(\alpha)$ in the continued
fraction \eqref{eq:continued-fraction} of a ``random'' real number, 
implies that the 
set of real numbers $\alpha>0$ for which \eqref{eq:lindeberg} holds has
Lebesgue measure $0$. Thus, the condition fails for a ``typical''
$\alpha$. Hence, as Beck observes, for a ``typical'' $\alpha$, the
interval discrepancy $\Delta(N,\alpha,I)$ does \emph{not} satisfy a
central limit theorem. 

Assuming the numbers $\lgt a$ in \eqref{eq:a-integer} behave like
``typical'' irrational numbers $\alpha$, we are thus led to the following
unexpected conjecture: 

\begin{conj}[Nonnormal Distribution of Benford Errors for Integer
Sequences $\{a^n\}$]
\label{conj:benford-error-non-CLT}
Let $a$ be any integer $\ge 2$ that is not a power of $10$, 
and let $d\in\{1,2,\dots,9\}$. 
Then the Benford error $E_d(N,\{a^n\})$ does \textbf{not} satisfy a 
central limit theorem in the sense of \eqref{eq:benford-error-CLT}.
\end{conj}

This conjecture, which is based on sound heuristics and thus seems
highly  plausible, represents a stunning turn-around in our quest to unravel the
mysteries behind Figure \ref{fig:histograms}.  If true, the conjecture
would imply that in \emph{none} of the cases shown 
in Figure \ref{fig:histograms} is the distribution asymptotically normal.
In particular, the seven distributions in Figure
\ref{fig:histograms} that seemed close to a normal distribution
and which appeared to be the most likely candidates for 
a ``real'' phenomenon  are now being revealed as the (likely) ``fakes'':
The observed normal shapes are (likely) mirages and  manifestations of Guy's
``strong law of small numbers.''

In light of this conjecture, it is natural to ask why the distributions
observed in Figure \ref{fig:histograms} appeared to have a normal
shape.  We believe there are two phenomena at work. For one, the
number of ``relevant'' continued fraction terms $a_i$ in the
approximation \eqref{eq:beck-heuristic} of $\Delta(N,\alpha,I)$ can be
expected to be around $\log N$ for most $\alpha$. Thus, even for values
$N$ on the order of one billion, the number of terms in the
approximating sum of random variables on the right of
\eqref{eq:beck-heuristic} may be too small to reliably represent the
long-term behavior of these sums.  Furthermore, while, for a ``typical''
$\alpha$, the ratio $a_k^2/(a_1^2+\dots+a_k^2)$ appearing in
\eqref{eq:lindeberg} is bounded away from $0$ for infinitely many values
of $k$, these values of $k$ form a very sparse set of integers,  
while for ``most'' $k$, the above ratio remains small.  This would suggest
that, even for numbers $\alpha$ that do not satisfy \eqref{eq:lindeberg},
$\Delta(N,\alpha,I)$ can be expected to be approximately normal ``most of
the time.''

\section{Concluding Remarks.}
\label{sec:concluding-remarks}

While our original goal of getting to the bottom of the numerical
mysteries in Table \ref{table:benford-1bil} and Figure
\ref{fig:histograms} and understanding the underlying general phenomenon
has been largely accomplished, the story does not end here. 
The results and conjectures obtained suggest a variety of generalizations,
extensions, and related questions.  

One can consider other notions of a ``perfect hit,'' such as
situations where rounding the Benford prediction \emph{up or down} always
gives the exact count, without insisting on the same type of rounding as
we have done in our definition of a (lower or upper) perfect hit.  For
sequences of the form $\{a^n\}$, these cases can be completely
characterized using the methods of this article. For example, the
digit-sequence pair $(d,\{a^n\})=(1,\{5^n\})$, which was one of the
potential perfect hits in Table \ref{table:benford-1bil},  
turns out to be a true perfect hit in the above more relaxed sense, but not  
in the sense of our definition. 

One can consider leading digits with respect to more general bases than
base $10$.  The Benford distribution \eqref{eq:benford}
has an obvious generalization for leading  
digits with respect to an arbitrary integer base
$b\ge 3$: simply replace the probabilities $P(d)=\lgt(1+1/d)$,
$d=1,\dots,9$, in \eqref{eq:benford} by the probabilities 
$P_b(d)=\log_b(1+1/d)$, $d=1,\dots,b-1$.  We have focused here on the base
$10$ case for the sake of exposition, but we expect  that all of our
results and conjectures can be extended to more general bases $b$.

One can ask if similar results hold for more general classes of sequences
than the geometric sequences we have considered here. We expect this to be
the case for ``generic'' sequences defined by linear recurrences---including
the Fibonacci and Lucas sequences---because solutions to
such recurrences can be expressed as  linear combinations of geometric
sequences, and it seems reasonable to expect that the leading digit behavior
of such a linear combination is determined by that of the ``dominating''
geometric sequence involved.

One can seek to more directly tie the behavior of the Benford error to
that of the continued fraction expansion of $\alpha=\lgt a$.  For example,
the heuristic of Beck described in Section \ref{sec:unbounded-error}
(see \eqref{eq:beck-heuristic}) suggests that it might be possible to relate
the size and behavior of the Benford error $E_d(N,\{a^n\})$ over a
\emph{specific} range for the numbers $N$ to the size and behavior of the
continued fraction terms $a_k$ for a corresponding range of indices $k$.

Finally, one can investigate other measures of ``surprising'' accuracy
of the Benford prediction. A particularly interesting one is provided by
``record hits'' of the Benford prediction,  defined as cases where the
Benford error at index $N$ is smaller in absolute value than at any
previous index. 
Heuristic arguments,
as well as numerical experiments we have carried out, 
suggest that these indices $N$ are
closely tied to the denominators in the continued fraction expansion of
$\lgt a$.

\begin{acknowledgment}{Acknowledgments.}
We are grateful to the referees for their careful reading of the paper and 
helpful suggestions and comments.  This work originated with an undergraduate
research project carried out in 2016 at the \emph{Illinois Geometry Lab}
(IGL) at the University of Illinois; we thank the IGL for providing this
opportunity.  
\end{acknowledgment}



\begin{biog}

\item[Zhaodong Cai]
received his B.S. in mathematics from the University of
Illinois in 2017 and is currently a Ph.D student at the University of
Pennsylvania. When not studying mathematics, he likes solving chess
puzzles.
\begin{affil}
Department of Mathematics,
University of Pennsylvania,
David Rittenhouse Lab,
209 South 33rd St.,
Philadelphia, PA 19104\\
zhcai@sas.upenn.edu
\end{affil}

\item[Matthew Faust]
is a Ph.D. student studying mathematics at Texas A\&M
University. He received B.S. degrees in computer engineering and
mathematics in 2018 from the University of Illinois. He is interested in
algebraic combinatorics and algebraic geometry. In his free time he enjoys
strategy games, preferably cooperative with Yuan Zhang.

\begin{affil}
Department of Mathematics,
Texas A\&M University, 
Mailstop 3368,
College Station, TX 77843\\
mfaust@math.tamu.edu
\end{affil}

\item[A.~J. Hildebrand] 
received his Ph.D. in mathematics
from the University of Freiburg in 1983 and has been 
at the University of Illinois since 1986, becoming professor emeritus in 2012.
Since retiring from the University of Illinois, he has supervised over one hundred
undergraduates on research projects in pure and applied mathematics and allied
areas. The present article grew out of one of these projects.

\begin{affil}
Department of Mathematics,
University of Illinois, 
1409 W. Green St., 
Urbana, IL 61801\\
ajh@illinois.edu
\end{affil}

\item[Junxian Li]
received her Ph.D. in 2018 from the University of Illinois 
under the supervision of Alexandru Zaharescu. Her
research interests are in number theory. During her Ph.D. studies, she has enjoyed
doing research with enthusiastic undergraduates as a graduate mentor in
the Illinois Geometry Lab. She is currently a postdoc at the Max Planck
Institute for Mathematics.
\begin{affil}
Max Planck Institute for Mathematics,
Vivatsgasse 7, D-53111 Bonn, Germany\\
jli135@mpim-bonn.mpg.de
\end{affil}

\item[Yuan Zhang] 
began his college education at the University of Illinois
as a major in natural resources and environment sciences. He soon  
changed his major to mathematics, receiving his B.S. degree in 2018. 
He is currently a Ph.D. student studying mathematics at the
University of Virginia.  He is interested in algebraic topology and
algebraic combinatorics. In his free time, he enjoys playing strategy
games, especially with his friend Matt Faust.
\begin{affil}
Department of Mathematics, 
University of Virginia, 
141 Cabell Dr.,
Kerchof Hall,
Charlottesville, VA 22904\\
yz3yq@virginia.edu
\end{affil}

\end{biog}

\end{document}